\documentclass[11pt]{article}
\usepackage{amsfonts,amsmath,amssymb,amsthm, amscd, subcaption}
\usepackage{graphicx}
\numberwithin{equation}{section}
\newtheorem{theorem}{Theorem}[section]
\newtheorem{prop}[theorem]{Proposition}

\newtheorem{cor}[theorem]{Corollary}

\theoremstyle{definition}
\newtheorem{definition}[theorem]{Definition}
\newtheorem{example}[theorem]{Example}
\newtheorem{remark}[theorem]{Remark}

\def\<{{\langle}}
\def\>{{\rangle}}

\def\a{{\alpha}}
\def\b{{\beta}}
\def\g{{\gamma}}

\def\Z{\mathbb Z}
\def\Q{\mathbb Q}
\def\R{\mathbb R}
\def\T{{\mathbb T}}
\def\S{{\mathbb S}}
\def\C{{\cal C}}
\def\B{\cal B}

\def\F{\mathbb F}
\def\a{\alpha}
\def\b{\beta}
\def\si{\sigma}
\def\t{\tau}
\def\k{{\kappa}}

\def\L{\cal L}
\def\La{\Lambda}
\def\M{{\cal M}}
\def\Rd{{\cal R}_d}
\def\s{{\bf s}}
\def\e{\epsilon}

\def\De{{\Delta}}

\def\ni{\noindent} 
\begin{document}

\title{Vertex-Colored Graphs, Bicycle Spaces and Mahler Measure}

\author{Kalyn R. Lamey
\and
Daniel S. Silver 
\and Susan G. Williams\thanks {The second and third authors are partially supported by the Simons Foundation.} }

\maketitle %{\setlength{\linewidth}{2in}

%%%%%%%%%%%%%%%%%%%%%%%%%%%%%% 

\begin{abstract} The space $\C$ of conservative vertex colorings (over a field $\F$) of a countable, locally finite graph $G$ is introduced.
When $G$ is connected, the subspace $\C^0$ of based colorings is shown to be isomorphic to the bicycle space of the graph. For graphs $G$ with a cofinite free 
$\Z^d$-action by automorphisms, $\C$ is dual to a finitely generated module over the polynomial ring $\F[x_1^{\pm 1}, \ldots, x_d^{\pm 1}]$ and for it polynomial invariants, the Laplacian polynomials $\De_k, k \ge 0$, are defined. Properties of the Laplacian polynomials are discussed. The logarithmic Mahler measure of $\De_0$ is characterized in terms of the growth of spanning trees.

\end{abstract}

\section{Introduction} \label{Intro} Graphs have been an important part of knot theory investigations since the nineteenth century. In particular, finite plane graphs correspond to alternating links via the medial construction (see section \ref{planegraphs}). The correspondence became especially fruitful in the mid 1980's when the Jones polynomial renewed the interest of many knot theorists in combinatorial methods while at the same time drawing the attention of mathematical physicists. 

Coloring methods for graphs also have a long history, one that stretches back at least as far as Francis Guthrie's Four Color conjecture of 1852. By contrast, coloring techniques in knot theory are relatively recent, mostly motivated by an observation in the 1963 textbook, \textit{Introduction to Knot Theory}, by Crowell and Fox \cite{CF63}. In view of the relationship between finite plane graphs and alternating links, it is not surprising that a corresponding theory of graph coloring exists. This is our starting point. However, by allowing nonplanar graphs and also countably infinite graphs that are locally finite, a richer theory emerges. 

Section \ref{conservativevertexcolorings} introduces the space $\C$ of conservative vertex colorings of a countable locally-finite graph $G$.
We identify the subspace $\C^0$ of based conservative vertex colorings with the bicycle space $\B$ of $G$. In section \ref{conservativeedgecolorings} we define the space of conservative edge colorings of $G$, which we show is naturally isomorphic to $\C^0$. When $G$ is embedded in the plane, yet a third type of coloring, coloring vertices and faces of $G$, is possible. The resulting space, called the Dehn colorings of $G$, is shown to be isomorphic to the space $\C$. We use it to extend and sharpen the known result that residues of the medial link components generate $\B$.

When $G$ admits a \emph{cofinite} free $\Z^d$-action by automorphisms (that is, an action with finite quotient), the vertex coloring space is dual to a finitely generated module over the polynomial ring $\F[x_1^{\pm 1}, \ldots, x_d^{\pm 1}]$. Techniques of commutative algebra are used to define  \textit{Laplacian polynomials} $\De_k, k \ge 0$, of $G$. In sections \ref{graphswithfreeZsymmetry} and \ref{graphswithfreeZ2symmetry} we consider infinite graphs that admit cofinite free $\Z^d$-symmetry, $d=1,2$. One may think of such a graph as the lift to the universal cover of a finite graph $\overline{G}$ in the annulus or torus. When $d=1$ and $\F$ is the 2-element field, we prove that the degree of the first nonzero polynomial $\De_k$ is twice the number of noncompact components of the medial graph of $G$; when $\F = \Q$, the degree of $\De_0$ is shown to be the minimum number of vertices that must be deleted from $\overline{G}$ in order to enclose the graph in a disk.

In section \ref{mahler} we enter the realm of algebraic dynamical systems. The compact group of conservative vertex colorings with elements of the circle $\R/\Z$ is a dynamical system with $d$ commuting automorphisms that is dual to  a finitely generated module over $\Z[x_1^{\pm 1}, \ldots, x_d^{\pm 1}]$. Using a theorem of D. Lind, K. Schmidt and T. Ward \cite{LSW90}, \cite{Sc95}, we characterize the logarithmic Mahler measure of $\De_0$ in terms of the growth of the number of spanning trees. 
This characterization was previously shown for connected graphs, first by R. Solomyak \cite{So98} in the case where the vertex set is $\Z^d$ and then for more general vertex sets  by R. Lyons \cite{Ly05}.  The algebraic dynamical approach here gives substantially simplified proofs, and it explains why Mahler measure appears.  The growth rate of spanning trees, also called the thermodynamic limit or bulk limit, has  been computed independently for many classic examples using purely analytical methods involving partition functions (cf.  \cite{Wu77}, \cite{SW00}, \cite{CS06}, \cite{TW10}.)   The coincidence of these values with Mahler measure was observed in \cite{GR12}.  \bigskip

We are grateful to Oliver Dasbach, Iain Moffatt and Lorenzo Traldi for their comments and suggestions, and to Doug Lind for drawing our attention to Solomyak's work. 

\section{Conservative vertex colorings} \label{conservativevertexcolorings}

Throughout, $G$ is assumed to be a countable locally finite graph. We denote the field of $p$ elements by $GF(p)$.

\begin{definition} A \emph{vertex coloring} of $G$ is an assignment of elements (called \emph{colors}) of a field $\F$ to the vertices of $G$. A vertex coloring is \emph{conservative} if, for every vertex $v \in V$, the \emph{Laplacian vertex condition} holds: \medskip
\begin{equation} \label{LVC} d \cdot \a = \sum_{i=1}^d \a_i, \end{equation}
where  $d$ is the degree of $v$, $\a$ is the color assigned to $v$, and $\a_1, \ldots, \a_d$ are the colors assigned to vertices adjacent to $v$, counted with multiplicity in case of multiple edges. (A self-loop at $v$ is counted as two edges from $v$ to $v$.)\end{definition} 

The set of vertex colorings will be identified with $\F^V$, where $V$ is the vertex set of $G$. Conservative vertex colorings of $G$ form a vector space under coordinate-wise addition and scalar multiplication.  We denote the vector space by $\C$. 

We will say two vertex colorings are \emph{equivalent} if they differ by a constant vertex coloring. The constant vertex colorings are a subspace of $\C$. We denote the quotient space by $\C^0$, and refer to it as the space of \emph{based vertex colorings} of $G$. If $G$ is connected and a \emph{base vertex} of $G$ is selected, then any element of $\C^0$ is uniquely represented by a vertex coloring that assigns zero to that vertex.  In this case, $\C^0$ can be regarded as a subspace of $\C$. 

Assume that the vertex and edge sets of $G$ are $V=\{v_1, v_2,\ldots \}$ and $E=\{e_1, e_2, \ldots\}$, respectively. The \emph{Laplacian matrix} $L= (L_{ij})$ is $D-A$, where 
$D$ is the diagonal matrix of vertex degrees, and $A$ is the adjacency matrix.
%that is, $L_{ij}$ is equal to minus the number of edges in $G$ between $v_i$ and $v_j\  (i\ne j)$ while $L_{ii}$ is the degree of $v_i$. 

If $V$ is infinite, then $L$ is a countably infinite matrix. Nevertheless, each row and column of $L$ has only finitely many nonzero terms, and hence many notions from linear algebra remain well defined. In particular, we can regard $L$ as an endomorphism of the space $\F^V$ of vertex colorings of $G$.  The following is immediate.

\begin{prop} The space $\C$ of conservative vertex colorings of $G$ is the kernel of the Laplacian matrix $L$. 

\end{prop} 

We orient the edges of $G$. 
The \emph{incidence matrix} $Q = (Q_{ij})$ is defined by 
\begin{equation*} Q_{ij}=  \begin{cases} 1 & \text{if\ } e_j  \text{\ terminates at\ } v_i,\\ -1 & \text{if\ } e_j
\text{\ originates at\ } v_i, \\ 0 & \text{\ else\ } \end{cases} \end{equation*}

One checks that $L = Q Q^T$, regardless of the choice of orientation.  

\begin{definition} The \emph{cut space} or \emph{cocycle space} $W$ of $G$ (over $\F$) is the subspace of $\F^E$ consisting of all, possibly infinite, linear combinations of row vectors of $Q$. \end{definition}

We regard $Q^T$ as a linear mapping from $\F^V$ to $\F^E$. It is well known and easy to see that if $G$ is connected, then the kernel of the mapping is 1-dimensional, and consists of the constant vertex colorings.

Denote by $\cdot$ the standard inner product on $\F^E$; that is,  $u \cdot w = \sum u_i w_i$,  whenever only finitely many summands are nonzero. 

\begin{definition} The \emph{cycle space} of $G$ (over $\F$) is the space of all $u \in \F^E$ with $u\cdot w =0$ for every $w\in W$ for which the product is defined.  We denote this space by $W^\perp$.
The \emph{bicycle space} $\B$ is $W \cap W^\perp$. \end{definition}

A vector $w \in \F^E$ is in the cut-space $W$ if and only if it is in the image of $Q^T$. Moreover, $w \in W^\perp$ if and only if $Qw = 0$. Since the kernel of $Q$ consists of the constant vertex colorings of $G$ when the graph is connected, we have: 

\begin{prop} \label{main} Assume that $G$ is connected. The based vertex coloring space $\C^0$ is isomorphic to the bicycle space $\B$. \end{prop} 

\begin{remark}   When $\F = GF(2)$, vectors $w \in \F^E$ correspond to subsets of $E$: an edge is included if and only if its coordinate in $w$ is nonzero. Vectors $w$ such that $Qw =0$ correspond to subgraphs in which every vertex has even degree. Bonnington and Richter \cite{BR03} call any such subgraph a \emph{cycle} of $G$, which explains the term ``cycle space." The term ``bicycle" arises since cycles in $\B$ are also cocycles. 

The reader is warned that there is not complete agreement in the literature about the definition of ``cycle" for infinite graphs. (See for example \cite{RV08}.)

\end{remark} 
%%%%%%%%%%%%%%%
\section{Conservative edge colorings} \label{conservativeedgecolorings} An element of $\F^E$ may be regarded as a coloring of the edge set.  Bicycles of connected graphs are edge colorings that satisfy two conservation laws, as we now show. We first orient $G$. 

\begin{definition} An \emph{edge coloring} is an assignment $\b$ of colors to the  edges of an oriented graph $G$.
\medskip

\noindent (1) The edge coloring $\b$ satisfies the \emph{cycle condition} if for every closed cycle
$e_1, \ldots, e_m$ in $G$ we have 

\begin{equation*}\label{CC} \sum \e_i \b(e_i)=0, \end{equation*} \medskip
where $\e_i=1$ if the edge $e_i$ is traversed in the preferred direction, and $\e_i=-1$ otherwise. \medskip

\noindent (2) The edge orientation $\b$ satisfies the \emph{Kirchoff vertex condition} if for every vertex $v$ with incident edges $e_1, \ldots, e_n$, 
\begin{equation*}\label{KVC} \sum \eta_i \b(e_i) =0, \end{equation*} \end{definition} \medskip
\noindent where $\eta_i =1$ if $v$ is the terminal vertex of $e_i$, and $\eta_i =-1$ if $v$ is the initial vertex.  \medskip

\noindent (3) A edge coloring is \emph{conservative} if it satisfies both the cycle condition and the Kirchoff vertex condition. 
\bigskip

An element $\beta=Q^T\alpha$ of the cut space $W$ assigns to an edge directed from $v_i$ to $v_j$ the color $\a(v_j)-\a(v_i)$.  Such a coloring clearly satisfies the cycle condition.  Conversely, suppose $\beta\in \F^E$ satisfies the cycle condition.  We can assign an arbitrary color to a basing vertex in each connected component of $G$, and extend along a spanning tree to obtain a  vertex coloring $\alpha$ with $\beta=Q^T \alpha$.  The cycle condition ensures that edges not on the spanning tree receive the right colors.

An edge coloring $\beta$ satisfies the Kirchhoff vertex condition if and only if $Q\beta $ is trivial, that is, $\beta\in W^\perp$. In summary, we have: 

\begin{prop} Let $G$ be a locally finite oriented graph. The set of conservative edge colorings of $G$ is equal to the bicycle space $\B$ of the graph.  \end{prop}

%%%%%%
\begin{remark} An equivalent theory of face colorings can also be defined. Compare with \cite{CSW14}.  \end{remark}

%%%%%%%%%%%%%%%%%%%%%%%%%%%%%%%%

\section{Plane graphs} \label{planegraphs} By a \emph{plane graph} we mean a graph $G$ embedded in the plane without accumulation points. Then $G$ partitions the plane into faces, some of which might be non-compact. 

\begin{definition} Let $G$ be a locally finite plane graph. A \emph{Dehn coloring} of $G$ is an assignment of colors to the vertices and faces of $G$ such that 
a chosen \emph{base face} is colored zero, and at any edge the \emph{conservative coloring condition} is satisfied: 
\begin{equation}\label{CCC} \a_1+ \g_1=\a_2 + \g_2, \end{equation} 
where $\a_1, \g_1, \a_2, \g_2$ are the colors assigned to vertices and faces $v_1, R_1, v_2, R_2$, as in Figure \ref{coloring}. 
\begin{figure}
\begin{center}
\includegraphics[height=1 in]{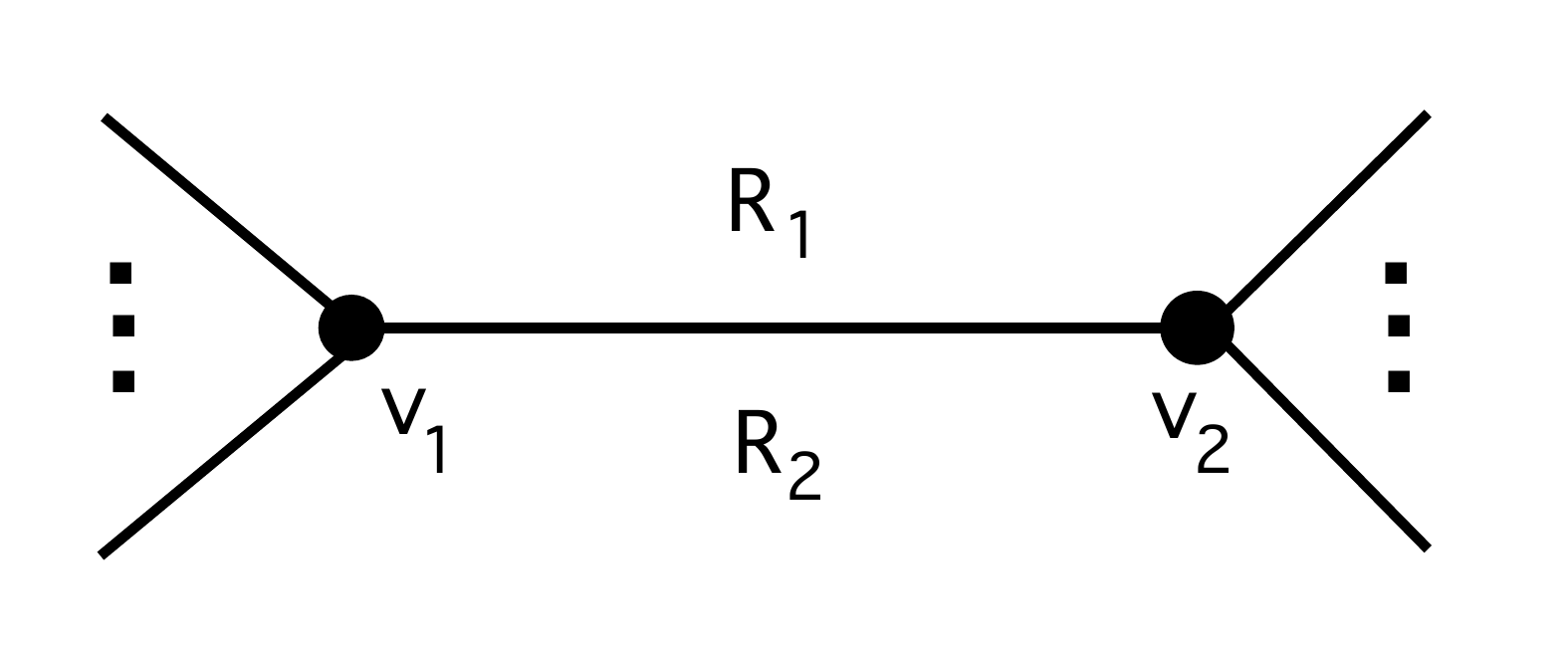}
\caption{Dehn Coloring Condition}
\label{coloring}
\end{center}
\end{figure}
\end{definition}

Like conservative vertex colorings of $G$, Dehn colorings form a vector space. We will denote it by $\C'$. It is not difficult to see that the underlying vertex assignment of any Dehn coloring is a conservative vertex coloring. (If we add the equations (\ref{CCC}) for edges adjacent to any vertex, then the face colors $\g_i$ cancel in pairs.) Let $r$ denote the restriction from $\C'$ to the vector space $\C$ of conservative vertex colorings.

\begin{prop} The  mapping  $r$ is an isomorphism from $\C'$ to $\C$.  
\end{prop}

\begin{proof} Consider any conservative vertex coloring of $G$. Beginning at the base face of $G$, which is colored with zero, ``integrate" along any path to another face, a path that is transverse to edges and does not pass through any vertex. If we arrive at an uncolored face $F$ from a face colored $\g$ by crossing an edge with vertices labeled $\a$ and $\a'$, then assign to $F$ the color $\a- \a' +\g$, where $\a$ is the label on the vertex to the left of the path. Such an assignment is path independent if and only if integrating around any small closed path surrounding a vertex returns the same color with which we began. It is easy to check that this will be the case if and only if condition (\ref{LVC}) is satisfied at every vertex. In this way we obtain a map $e:\C \to \C'$. It is not difficult to see that $e$ is a homomorphism. That $r\circ e$ is the identity map on $\C$ is clear. The path-independence of integration implies that $e \circ r$ is the identity map on $\C'$. 
\end{proof} 

Medial graphs provide a bridge between planar graphs and links.
The \emph{medial graph} of a locally finite plane graph $G$ is the plane graph  $M(G)$ obtained from the boundary of a thin regular neighborhood of $G$ by pinching each edge to create a vertex of degree 4, as in Figure \ref{brings2}.

The medial graph $M(G)$ is the projection in the plane of an alternating link $\L$ (possibly with noncompact components if $G$ is infinite) called the \emph{medial link}. It is well defined up to replacement of every crossing by its opposite. By a \emph{component of the medial graph} we will mean the projection of a component of $\L$.  

\begin{figure}[ht]
\begin{minipage}[b]{0.45\linewidth}
\centering
\includegraphics[width=\textwidth]{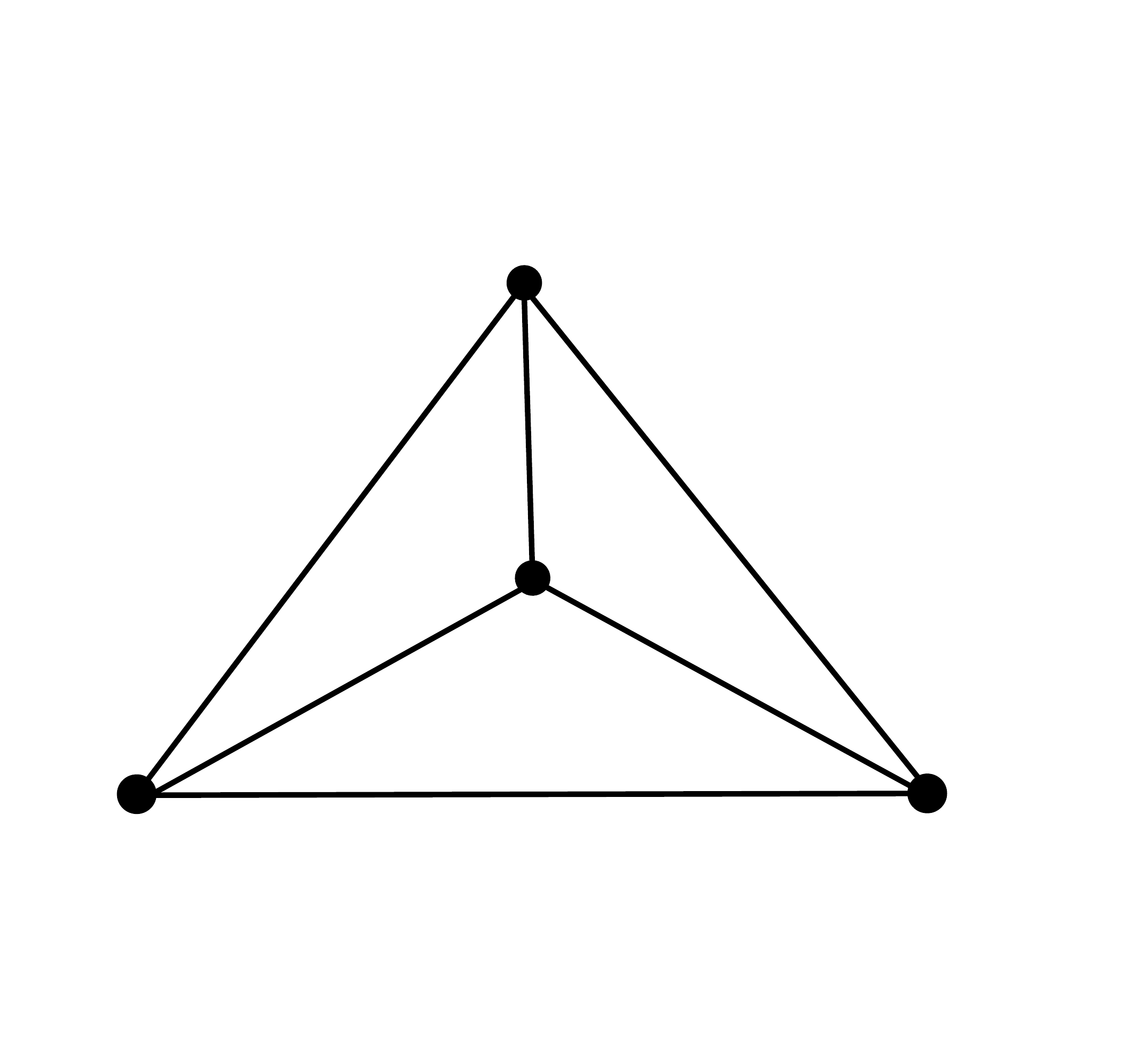}
\caption{Plane Graph $G$}
\label{fig:figure1}
\end{minipage}
\hspace{0.5cm}
\begin{minipage}[b]{0.45\linewidth}
\centering
\includegraphics[width=\textwidth]{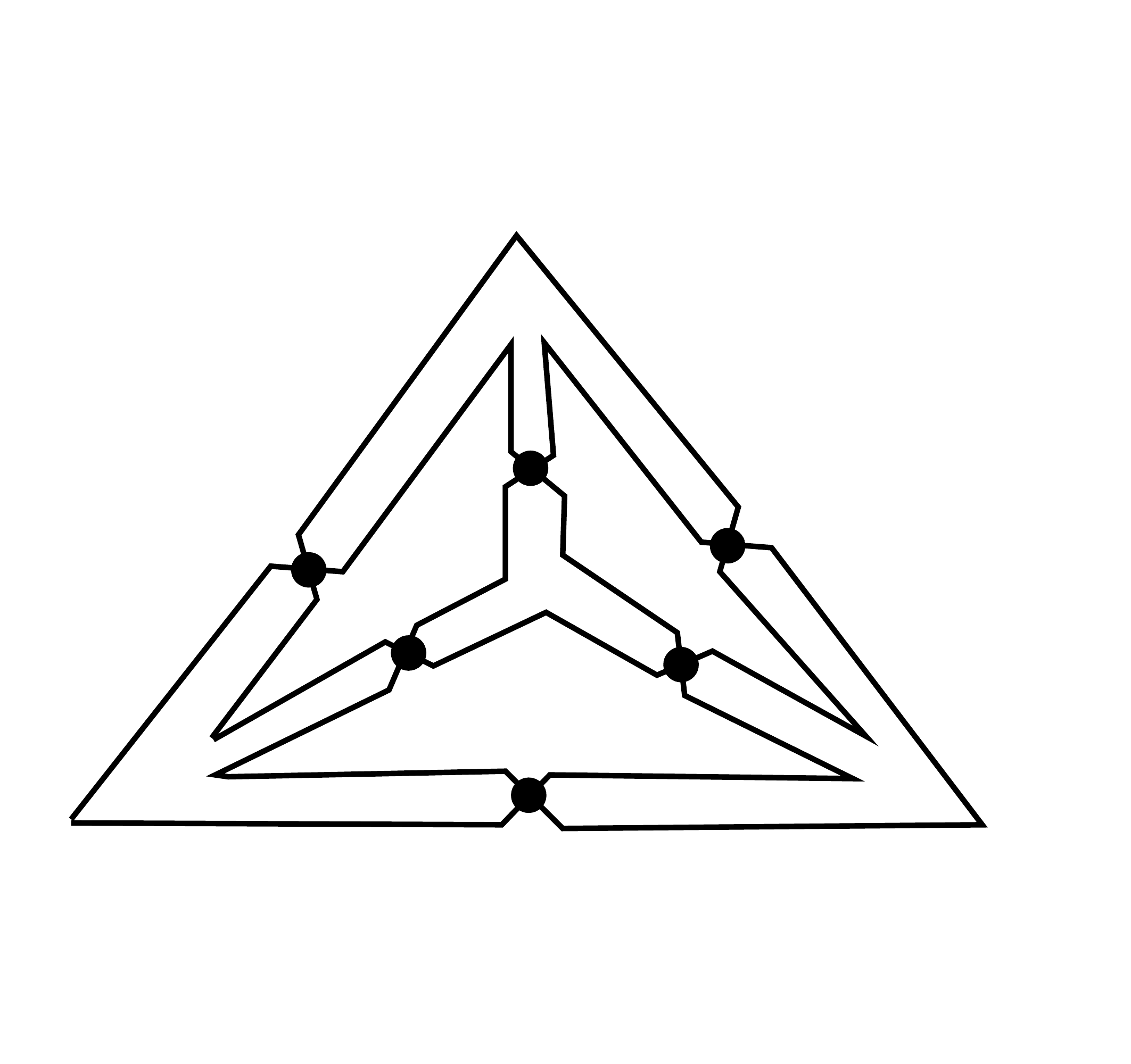}
\caption{Medial Graph $M(G)$}
\label{brings2}
\end{minipage}
\end{figure}

\begin{remark} (1) The vector space of Dehn colorings of $G$ is isomorphic to a vector space of  Dehn colorings of  the link $\L$ (see \cite{Ka83} \cite{CSW14}). It is well known that $\C^0$ is isomorphic to the first homology group of the 2-fold cover of $\L$ with coefficients in $\F$. Classes of lifts of meridians, one from all but one component of $\L$, comprise a basis. While much of our motivation derives from these facts, we do not make explicit use of them here.   \smallskip

(2) When $G$ is finite, $L$ is the Goeritz matrix of the  alternating link $\L$  associated to $G$. (See \cite{Tr85}.)

\end{remark}

%%%%%%%%%%%%%%%%%%%

For the remainder of the section, $\F = GF(2)$. In this case, edge orientations are not needed. Recall that  edge colorings correspond bijectively to subsets of $E$. 

Given a component of the medial graph $M(G)$, we consider the subset of $E$ consisting of those edges crossed by the component exactly once. Following \cite{My09} we refer to the edge set as the \emph{residue} of the component. 

In \cite{Sh75} H. Shank proved that for any finite plane graph $G$ the residues of components of $M(G)$ span $\B$. (Shank ascribed this to J.D. Horton.) An analogous result was shown in \cite{My09} for infinite, locally finite plane graphs, but where cycles have a more restrictive definition. 

In \cite{SW14}  the second and third authors gave a very short, elementary proof of the following theorem for finite plane graphs $G$ using an idea borrowed from knot theory. 
The argument here allows for the medial graph to have any number of components, including non-compact components. After finishing the argument, the authors became aware of \cite{BKR07}, in which a similar argument is used in the case that $G$ is finite. 

\begin{theorem} (cf. Theorem 17.3.5 of \cite{GR01}) Assume that $G$ is connected. For $\F=GF(2)$, the residues of all but one arbitrary component of the medial graph $M(G)$ form a basis for the bicycle space $\B$ of the graph $G$. \end{theorem}

\begin{proof} Regard colorings with $0, 1$ of components of $M(G)$ as the elements of a vector space $\M$, with addition and scalar multiplication defined in the obvious way. 
There is a natural isomorphism between $\M$ and the space $\C$ of conservative vertex colorings of $G$. Given an element of $\M$, integrate along paths from the base face in order to color all vertices and faces of $G$. Each time we cross a component of $M(G)$ colored with $1$, we change color, until we reach the desired vertex or face. (Here a \emph{face} is considered to be outside the thin regular neighborhood of $G$ by which we describe $M(G)$, as in Figure \ref{brings2}.) The component can be regarded as a smoothly immersed curve in the plane. By considering intersection numbers modulo 2 (or using the Jordan Curve Theorem) we see that the coloring  we obtain is path independent and hence well defined. 

The assignment defines a homomorphism from $\M$ to $\C$. An inverse homomorphism is easily defined. Consider any non-crossing point on a component of $M(G)$. On one side is the pinched regular neighborhood of $G$ containing a single vertex of the graph. On the other is a face of $G$. Assign to that component of $M(G)$  the sum of the colors of the vertex and face. The condition (\ref{CCC}) ensures that the assignment does not depend on the  point of the component. 

Like $\C$, the vector space $\M$ has a subspace of constant colorings. Let $\M^0$ denote the quotient space. The isomorphism from $\M$ to $\C$ induces one from $\M^0$ to $\C^0$. It can be made explicit by choosing a  \emph{base vertex} of $G$ near a \emph{base component} of $M(G)$, and requiring that each be colored with zero. 

By Proposition \ref{main}, $\C^0$ is isomorphic to $\B$, which we identified with the space of conservative edge colorings of $G$. Given a basic element of $\M^0$, a coloring that assigns $1$ to a non-base component and zero to the others, we see that the associated vertex coloring of $G$ assigns different colors to vertices of an edge precisely when the component intersects that edge exactly once. Since these are the edges colored $1$ by the associated edge coloring, the proof is complete. \end{proof}

We conclude the section with an example that illustrates most of the ideas so far. 
Examples involving infinite graphs appear in the next section. 

\begin{example} Consider the complete graph $G$ on four vertices embedded in the plane, as in Figure \ref{brings3}. The incidence matrix is: 
$$Q  = \begin{pmatrix} -1 & 0 & 1 & -1 & 0 & 0\\ 1 & -1 & 0 & 0 & -1& 0 \\
0 & 1 & -1 & 0 & 0 & -1 \\ 0 & 0 & 0 & 1 & 1& 1 \end{pmatrix}$$
while the Laplacian matrix is:
$$L = \begin{pmatrix} 3 & -1 & -1 & -1\\ -1 & 3 & -1 & -1 \\ -1 & -1 & 3 & -1\\ -1& -1& -1& 3 \end{pmatrix}. $$
We choose $v_1$ to be a base vertex. Then the bicycle space $\B$ is isomorphic to the nullspace of the principal minor 
$$ L_0= \begin{pmatrix} 3 & -1 & -1 \\ -1 & 3 & -1  \\ -1 & -1 & 3  \end{pmatrix}$$
with rows and columns corresponding to $v_2, v_3, v_4$. 
Since the determinant of $L_0$ is 16, $\B$ is trivial unless $\F$ has characteristic 2. 

When $\F = GF(2)$, the space $\C^0$ of based conservative vertex colorings is the nullspace of $L_0$, which has basis $(1, 1, 0)^T, (0, 1, 1)^T\in V^3$. Recall that each of these basis vectors represents a coloring of the vertices of $G$ using $0, 1$ with the chosen base vertex $v_1$ receiving zero. A basis vector for $\B$ is obtained from each by selecting the edges that have differently colored vertices. This gives $(1, 0, 1, 0, 1, 1)^T, (1,1,0,1,0,1)^T \in E^6$. (Alternatively, the first vector is the sum of the second and third rows of $Q$ while the second vector is the sum of the second and fourth.) The residues correspond to two components of the medial link $\L$. 

\begin{figure}
\begin{center}
\includegraphics[height=2 in]{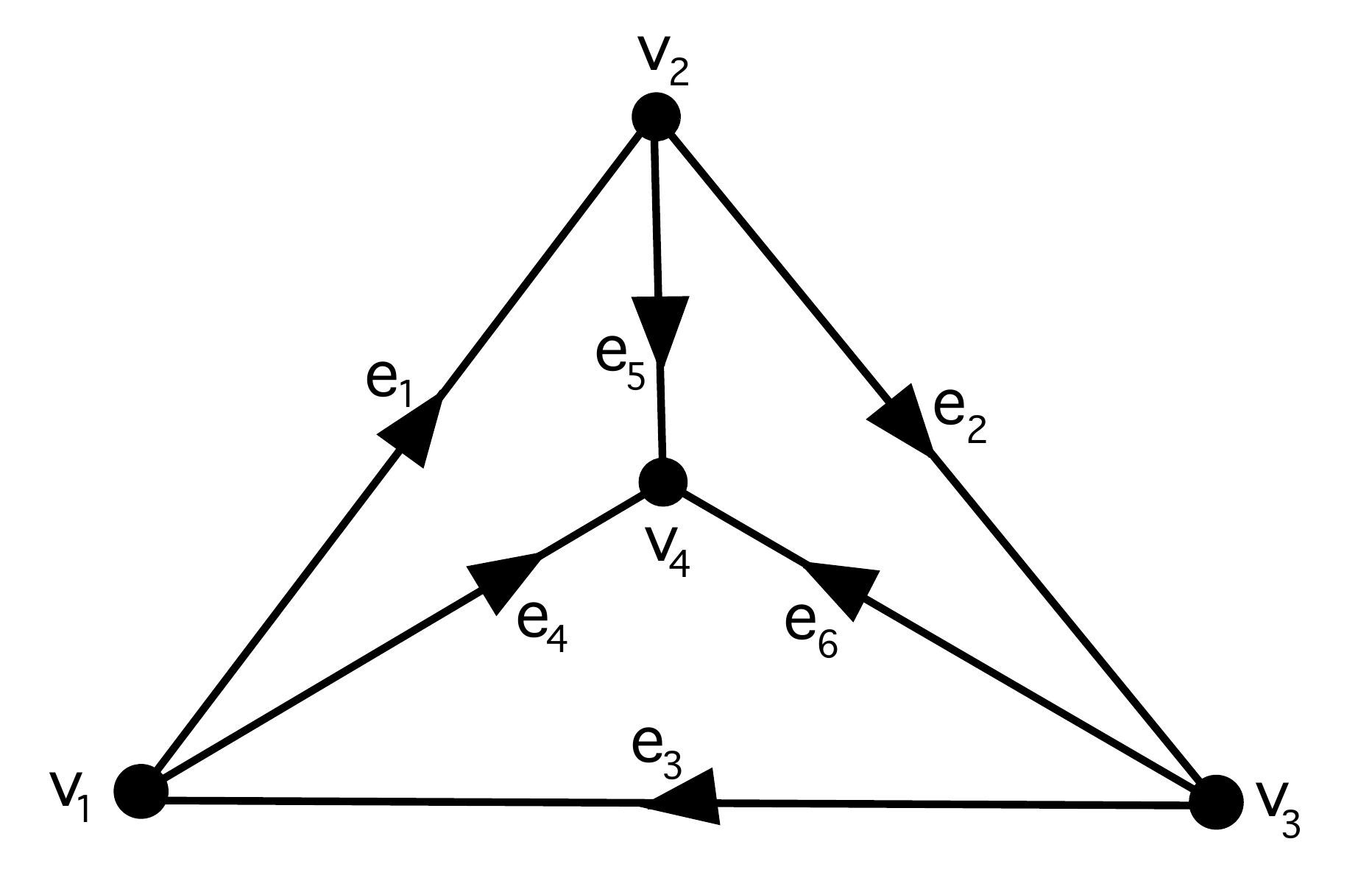}
\caption{Plane Graph $G$ with Vertices and Oriented Edges Labeled}
\label{brings3}
\end{center}
\end{figure}

\end{example} 
%%%%%%%%%%%%%%%%%%%%
\section{Graphs with free $\Z$-symmetry} \label{graphswithfreeZsymmetry} It is well known that when $G$ is a finite plane graph, the number of components of the associated medial graph is equal to the nullity of the Laplacian matrix over $GF(2)$. (See \cite{SW14} for a short proof using Dehn colorings.) In this section we consider graphs $G$ that arise by starting with a finite graph $\overline{G}$ with vertex set embedded in an annulus $A$ (we allow the edges to cross one another), and then lifting to the universal cover of $A$. If $\overline{G}$ is connected and not contained in any disk in $A$, then $G$ will be connected. Such a graph has a cofinite free $\Z$-symmetry; that is, $\Z$ acts freely on $G$ by automorphisms and the quotient graph $\overline{G}$ is finite.

We regard $\Z$ as a multiplicative group with generator $x$. Let $\overline{V} = \{v_1, \ldots, v_n\}$ and $\overline{E}= \{e_1, \ldots, e_m\}$ be the vertex and edge sets of $\overline{G}$. Choose a lift $v_{i, 0}$ of each $v_i\in \overline{V}$ and also a lift $e_{j, 0}$ of each $e_j \in \overline{E}$. We denote the images of $v_{i, 0}$ and $e_{j,0}$ under the group element $x^\nu \in \Z$ by $v_{i, \nu}$ and $e_{j, \nu}$, respectively. Then the vertex and edge sets $V, E$ of $G$ each consists of a finite family of vertices $v_{i, \nu}$ and $e_{j, \nu}$, respectively. 

Let $\F[x^{\pm 1}]$ denote the ring of Laurent polynomials in variable $x$ with coefficients in $\F$. The  space $\C$ of conservative vertex colorings is the dual $\text{Hom}(C, \F)$ of a finitely generated $\F[x^{\pm 1}]$-module $C$ with relation matrix $L=L(x)$. We may take $L$ to be the $n\times n$ matrix $L(x)=D-A(x)$, where $D$ is the diagonal matrix of degrees of $v_i$ and  $A(x)_{i,j}$ is the sum of $x^\nu$ for each edge in $G$ from $v_{i,0}$ to $v_{j,\nu}$.

For $k \ge 0$, let $\De_k(x)$ denote the $k$th elementary divisor of $\C$ (abbreviated by $\De_k$). It is the greatest common divisor of the 
determinants of all $(n-k) \times (n-k)$ minors of the Laplacian matrix $L$. It is well defined up to multiplication by units in $\F[x^{\pm 1}]$. We call $\De_k$ the $k$th \emph{Laplacian polynomial} of $G$. The polynomial $\De_0$ can be regarded as the determinant of the Laplacian matrix of $\overline{G}$ with additional, homological information about algebraic winding numbers of cycles in the annulus $A$ (see \cite{Fo93}, \cite{Ke11}). 

The \emph{degree} of a nonzero Laurent polynomial $f \in \F[x^{\pm 1}]$ is the difference of the maximum and minimum degrees of a nonzero monomial in $f$, denoted by $\text{deg}_\F\ f$.  

A polynomial $f \in \F[x^{\pm 1}]$ is \emph{reciprocal} if $f(x^{-1}) x^n= f(x)$ for some $n\in\Z$. 
Less formally, $f$ is reciprocal if its sequence of coefficients is palindromic. 

\begin{prop} Let $G$ be a graph with cofinite free $\Z$-symmetry.\smallskip

 \ni (1) For any $k \ge 0$, the Laplacian polynomial $\De_k$ is reciprocal. 
\smallskip

\ni (2) $\De_0(G)$ is  divisible by $(x-1)^2.$ \end{prop}

\begin{proof} (1) %After multiplying each row of $L$ by a suitable unit of $\F[x^{\pm 1}]$, we can assume without loss of generality that, for each $i$, the contribution of the diagonal matrix $D$ is a constant. 
If we choose $L$ as above then $L(x^{-1}) = L(x)^T$. Hence $\De_k(x)$ is reciprocal. \\

(2) If $\De_0$ is identically zero, then there is nothing to prove. Assume that $\De_0$ is nonzero. Setting $x=1$ makes the row sums of $L$ zero. Hence $\De_0(1)=0$. 

By (1), we can normalize so $\De_0(x)=\De_0(x^{-1})$.  Thus $\De_0$ is reciprocal of even degree, and its roots come in reciprocal pairs.  Hence 1 is a double root.
%Normalizing so $\De_0(x)=\De_0(x^{-1})$ and differentiating, 
%$$\De'_0(x)=\frac{d}{dx} \De_0(x^{-1}) = \De'_0(x^{-1})(-x^{-2}).$$
%\ni Hence $\De'_0(1) = -\De'_0(1)$. Since $\De'_0(1)$ vanishes, $x=1$ is a double root 
%of $\De_0(x)$, and so $(x-1)^2$ divides $\De_0(x).$
%[THIS WORKS FOR ANY RECIPROCAL POLY OF EVEN DEGREE.  Do the other $\De-k$ have even degree?]
\end{proof}

\begin{remark} \label{alexpoly} If $\overline G$ is embedded in the annulus, then by the medial graph construction, we obtain an alternating link $\ell= \ell_1 \cup \cdots \cup \ell_\mu$ in the thickened annulus, which we may regard as a solid torus $\S^1 \times D^2$. Consider the encircled link $\ell \cup m$, where $m$ is a meridian $ \{p\}\times \partial D^2$, with $p \in \S^1$. It is not difficult to show that the Laplacian polynomial $\De_0$ is equal to 
$(x-1)\De(-1, \ldots, -1, x)$, where $\De$ is the Alexander polynomial of $\ell \cup m$ with $\mu+1$ variables, and $x$ is the variable corresponding to $m$. We will not make use of this fact here.  
\end{remark} 

Now assume that the graph $\overline{G}$ is embedded in the annulus.  Then $G$ will be planar.  Recall that the medial graph $M(G)$ is the projection of an alternating link $\L$. Components of $\L$ are called components of the medial graph.

An \emph{annular cut set} of $\overline{G}$ is a set of vertices such that when the vertices and incident edges are removed, the resulting graph is contained in a disk neighborhood. The \emph{annular connectivity} $\k(\overline{G})$ is the minimal cardinality of an annular cut set of $\overline{G}$.

\begin{theorem} \label{degthm}

Let $G$ be a graph embedded in the plane with cofinite free $\Z$-symmetry. \smallskip

\ni (1) For $\F= GF(2)$, $\text{deg}_\F\,\De_s$ is equal to the number of noncompact components of $M(G)$, where $\De_s$ is the first
nonzero Laplacian polynomial of $G$.\smallskip

\ni (2) For $\F = \Q$,  $\text{deg}_\F\,\De_0= 2\k(\overline{G})$. \end{theorem}

\begin{remark} Example \ref{circulant} below shows that the hypothesis that $G$ is a plane graph cannot be eliminated in part (2).
\end{remark}

\begin{proof} (1) Recall that the components of the medial graph $M(G)$ correspond to a basis for the vector space $\C$ over $\F=GF(2)$. If $M(G)$ has closed components, then each $\Z$-orbit of these spans a free $\F[x^{\pm 1}]$-summand of the module $C$. The module $C$ decomposes as $\text{Tor}(C) \oplus \F[x^{\pm 1}]^s$, where $\text{Tor}(C)$ is the $\La$-torsion submodule of $C$, and $s$ is the number of $\Z$-orbits of closed components of $M(G)$.  The first non-vanishing Laplacian polynomial is $\De_s(x)$. Its degree, which is the dimension of $\text{Tor}(C)$ regarded as a vector space over $\F$, is the number of non-compact components of $M(G)$.  

(2) A combinatorial expression for $\De_0$ was given by R. Forman \cite{Fo93}. Another proof was later given by R. Kenyon \cite{Ke11} \cite{Ke12}, and we use his terminology. 
\begin{equation*}\label{forman} \De_0 (x) = \sum_{k=1}^\infty C_k(2-x-x^{-1})^k. \end{equation*}
Here $C_k$ is the number of cycle rooted spanning forests in the graph $\overline{G}$ having $k$ components and every cycle essential (that is, non-contractible in the annulus). A \emph{cycle rooted spanning forest} (CRSF) is a subset of edges of $\overline{G}$ such that (i) every vertex is an endpoint of some edge; and (ii) each component has exactly one cycle. 

It follows that the degree of $\De_0$ is twice the maximum cardinality $N$ of a set of pairwise disjoint essential cycles in $A$. We complete the proof by showing that $N$ is the annular connectivity $\k(\overline{G})$. 

Regard the annulus $A$ as lying in the plane, having inner and out boundary components.  Let $\g_1, \ldots, \g_n$ be a set of pairwise disjoint essential cycles of maximal cardinality, ordered so that $\g_i$ lies inside $\g_{i+1}$.  There is a path with interior in the complement of $\overline{G}$ from the inner boundary to some vertex $v_1$ of $\g_1$, since otherwise the obstructing edges would form an essential cycle disjoint from $\g_1, \ldots, \g_n$. Inductively, there is a path from $v_{i}$ on $\g_{i}$ to a vertex $v_{i+1}$ on $\g_{i+1}$, and from $v_n$ to the outer boundary, with interior in the complement of $\overline{G}$.  Then $\{v_{1,0}, \ldots, v_{n,0}\}$ is a vertex cut set for the graph $G$. No vertex cut set can have smaller cardinality, since its projection must have  a vertex on each of the cycles $\g_1, \ldots, \g_n$. 
\end{proof}

If the annular connectivity $\k(\overline{G})$ is equal to $1$ then $\overline{G}$ can be split at a vertex $v$, producing a graph $H$ with vertices $v', v''$ so that $G$ is an infinite join of copies $(H_\nu, v'_\nu, v''_\nu)$:
$$\cdots *H_\nu * H_{\nu+1}* \cdots,$$
 where $H_\nu$ is joined to $H_{\nu+1}$ along $v'_\nu \in H_\nu$ and $v''_{\nu+1} \in H_{\nu+1}$.

\begin{cor} If $\k(\overline{G})=1$, then the Laplacian polynomial $\De_0(x)$ of $G$ is equal to $\t(H)(x-1)^2$, where $\t(H)$ is the number of spanning trees in $H$.
\end{cor}
\begin{proof} Note that up to unit multiplication in $\Q[t^{\pm 1}]$, the term $(x-1)^2$ is equal to 
$(2-x-x^{-1})$. In the proof of Theorem \ref{degthm}, $C_k =0$ unless $k=1$. Every CRSF with an essential cycle becomes a spanning tree for $H$ when $\overline{G}$ is split along $v$ to produce $H$. Conversely, every spanning tree for $H$ becomes a CRSF with essential cycle when $v'$ and $v''$ are rejoined. Hence there is a bijection between the set of CRSFs with an essential cycle and the set of spanning trees for $H$.  
\end{proof}

\begin{figure}[ht]
\begin{center}
\begin{minipage}[b]{0.3\linewidth}
\centering
\includegraphics[width=\textwidth]{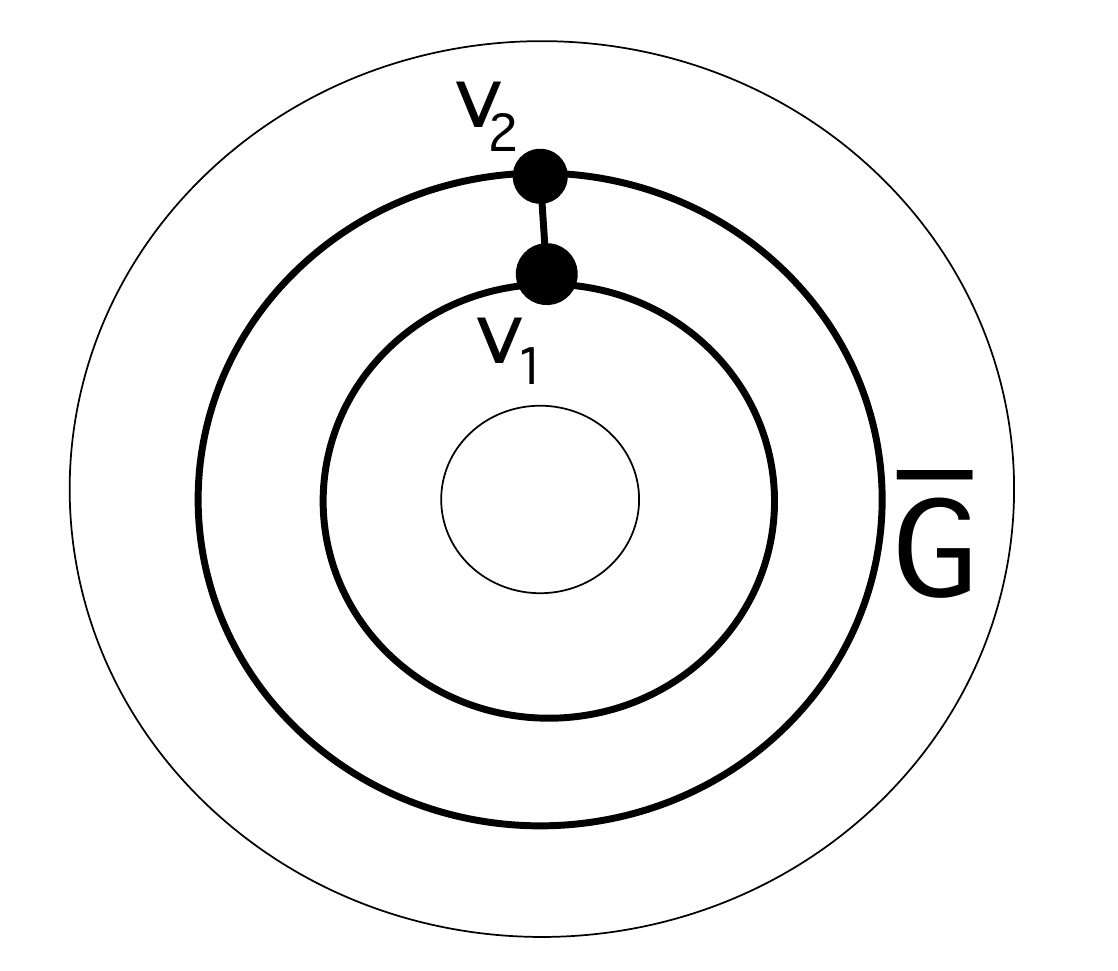}
\caption{Graph $\overline{G}$}
\label{ladderquot}
\end{minipage}
\hspace{0.4cm}
\begin{minipage}[b]{0.50\linewidth}
\centering
\includegraphics[width=\textwidth]{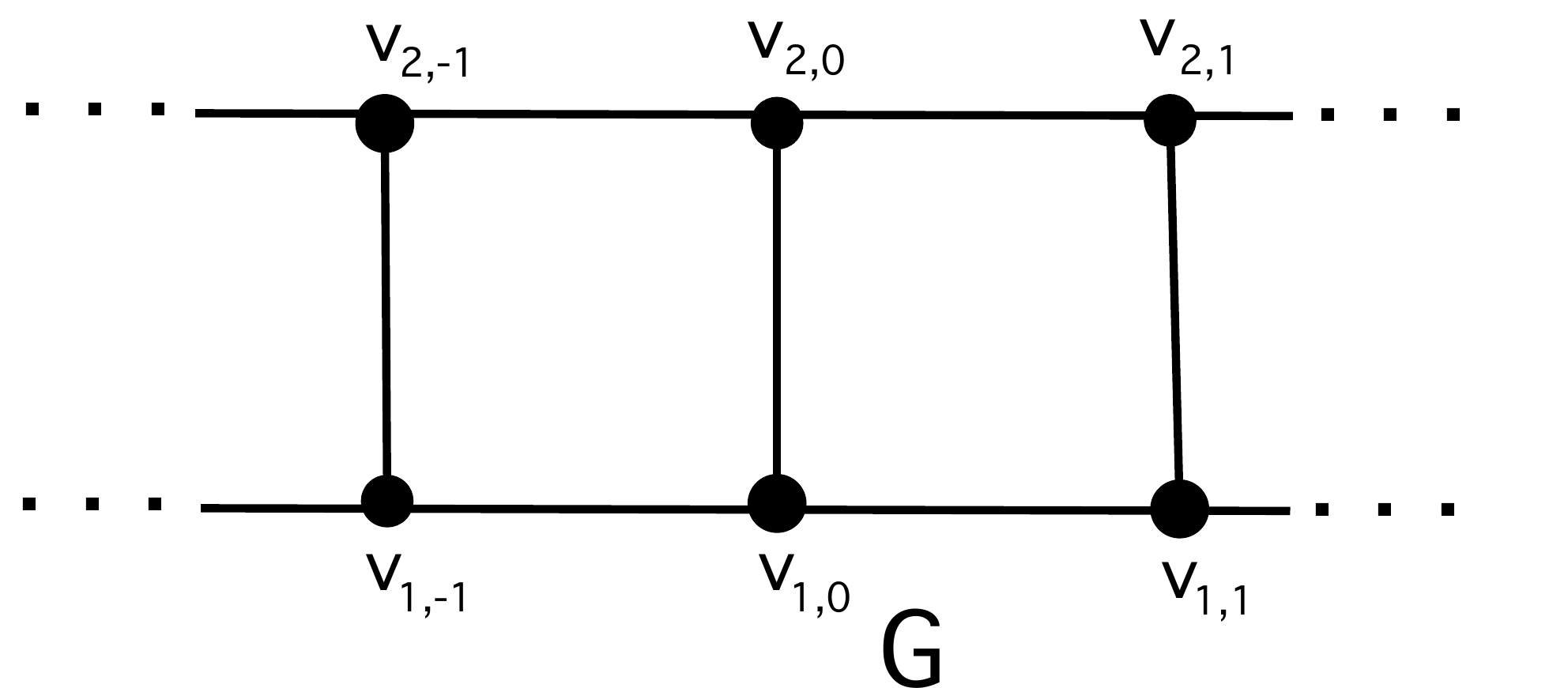}
\caption{Graph $G$}
\label{ladderquotient}
\end{minipage}
\end{center}
\end{figure}

\begin{example} \label{ladder}  Consider the graph $\overline{G}$  embedded in the annulus as in Figure \ref{ladderquot}. It lifts to the ``ladder graph" $G$, which appears in Figure
\ref{ladderquotient}. The Laplacian matrix is 
\begin{equation*} L= \begin{pmatrix} 3-x-x^{-1} & -1\\ -1 & 3 - x - x^{-1} \end{pmatrix} \end{equation*} 
The $0$th Laplacian polynomial $\De_0(x)$ is $(x-1)^2 (x^2-4 x+1)$, with degree 4 over  $\F=\Q$ or $\F=GF(2)$. The reader can check that $\{v_1, v_2\}$ is an annular cut set of $\overline{G}$ with minimal cardinality, and $M(G)$ has four components. The bicycle space $\B$ has dimension $3$ for any field $\F$.

\end{example}

\begin{example} Consider the ``girder graph" $G$ in Figure \ref{girder}. The Laplacian matrix is 
\begin{equation*} L= \begin{pmatrix} 6-2x-2x^{-1} & -1-x^{-1}\\ -1-x & 6 - 2x - 2x^{-1} \end{pmatrix}. \end{equation*}
With $\F=\Q$ the $0$th Laplacian  polynomial is $\De_0(x) = (x-1)^2 (4x^2-17x+4)$.
Again one can check that the quotient graph has an annular cut set of cardinality 2 but 
no cut set of smaller size.  

If $\F=GF(2)$, then $\De_0(x) = (x-1)^2$. It is an amusing exercise to verify that $M(G)$ has exactly two components. 

The bicycle space $\B$ has dimension $3$ whenever the characteristic of $\F$ is different from $2$. When the characteristic is $2$, $\B$ is $1$-dimensional. 

\begin{figure}
\begin{center}
\includegraphics[height=1.4 in]{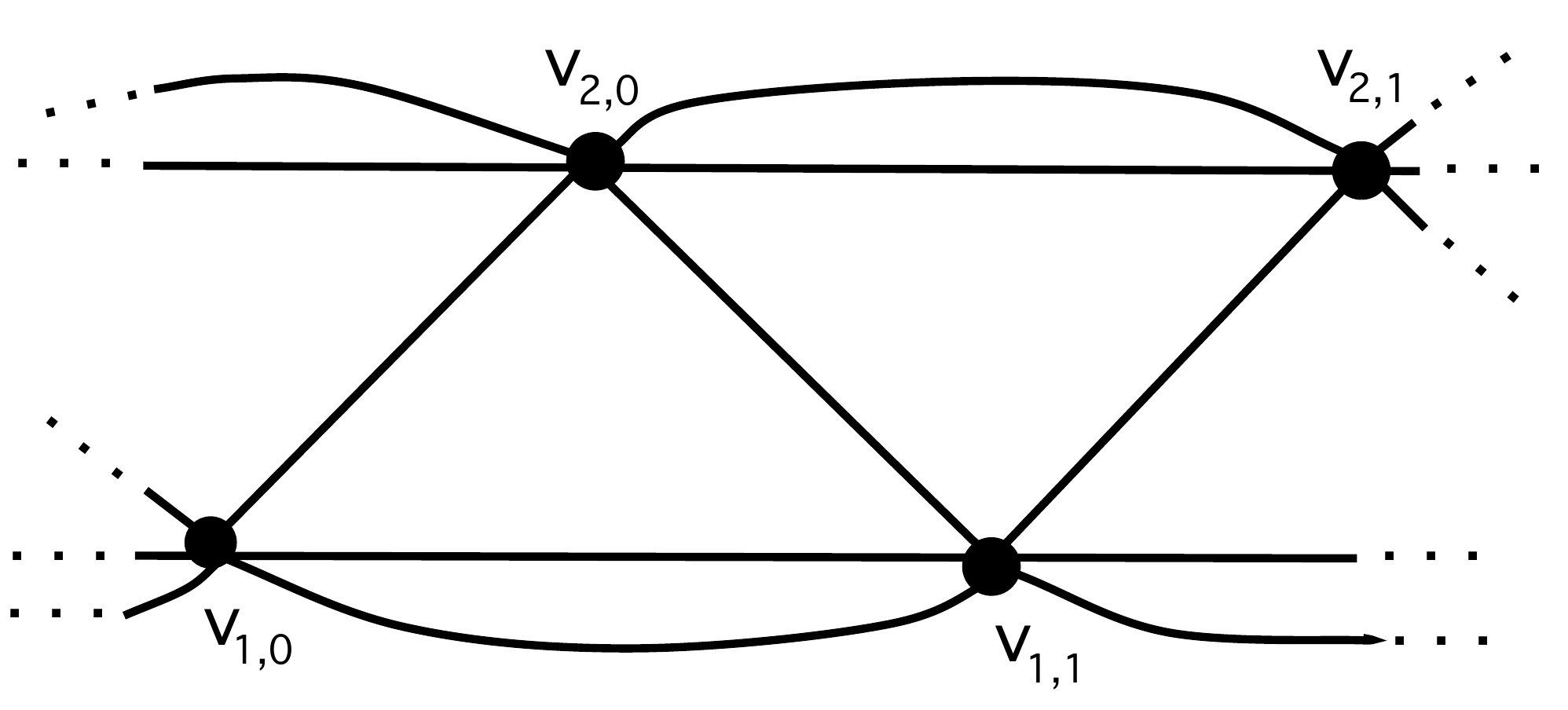}
\caption{Girder Graph $G$}
\label{girder}
\end{center}
\end{figure}

\end{example}

%%%%%%%%%%%%%%%%%%%%%%%%%

\section{Graphs with free $\Z^2$-symmetry}\label{graphswithfreeZ2symmetry}

 A finite graph $\overline{G}$ with vertex set embedded in the torus lifts to a graph $G$ with a $\Z^2$-symmetry. We regard $\Z^2$ as a multiplicative group with generators $x, y$ corresponding to a fixed meridian and longitude of the torus. Each vertex $v$ of $G$ is covered by a countable collection $v_{(i, j)}$ of vertices such that the action of $x^m y^n$ sends 
each $v_{(i, j)}$ to $v_{(i+m, j+n)}$, for $m, n \in \Z$.  We will assume that $G$ is connected.

As in the case of graphs with free $\Z$-symmetries, the space $\C$ of conservative vertex colorings of $G$ is the dual $\text{Hom}(C, \F)$ of a finitely generated module $C$, but here the ring is $\F[x^{\pm 1}, y^{\pm 1}]$. The Laplacian matrix $L$ is a relation matrix for $C$, and there is a sequence $\De_k(x, y)$ of  $k$th elementary divisors (abbreviated by $\De_k$), well defined up to multiplication by units. Again, we call $\De_k$ the $k$th Laplacian polynomial of $G$.

\begin{example}\label{2grid} Consider the simplest graph $\overline{G}$, having a single vertex $v$, embedded in the torus such that each face is contractible, as in Figure \ref{gridquotient}. It lifts to a plane graph $G$ with vertices $v_{(i, j)}$, as in Figure \ref{grid}.
The Laplacian matrix is a $1 \times 1$-matrix: 
\begin{equation*} L= (4-x-x^{-1}-y-y^{-1}) \end{equation*} 
and so the $0$th Laplacian polynomial is $\De_0(x, y) = 4-x-x^{-1}-y-y^{-1}$. 
The bicycle space is infinite-dimensional for any field $\F$.

\begin{figure}[ht]
\begin{center}
\begin{minipage}[b]{0.40\linewidth}
\centering
\includegraphics[width=\textwidth]{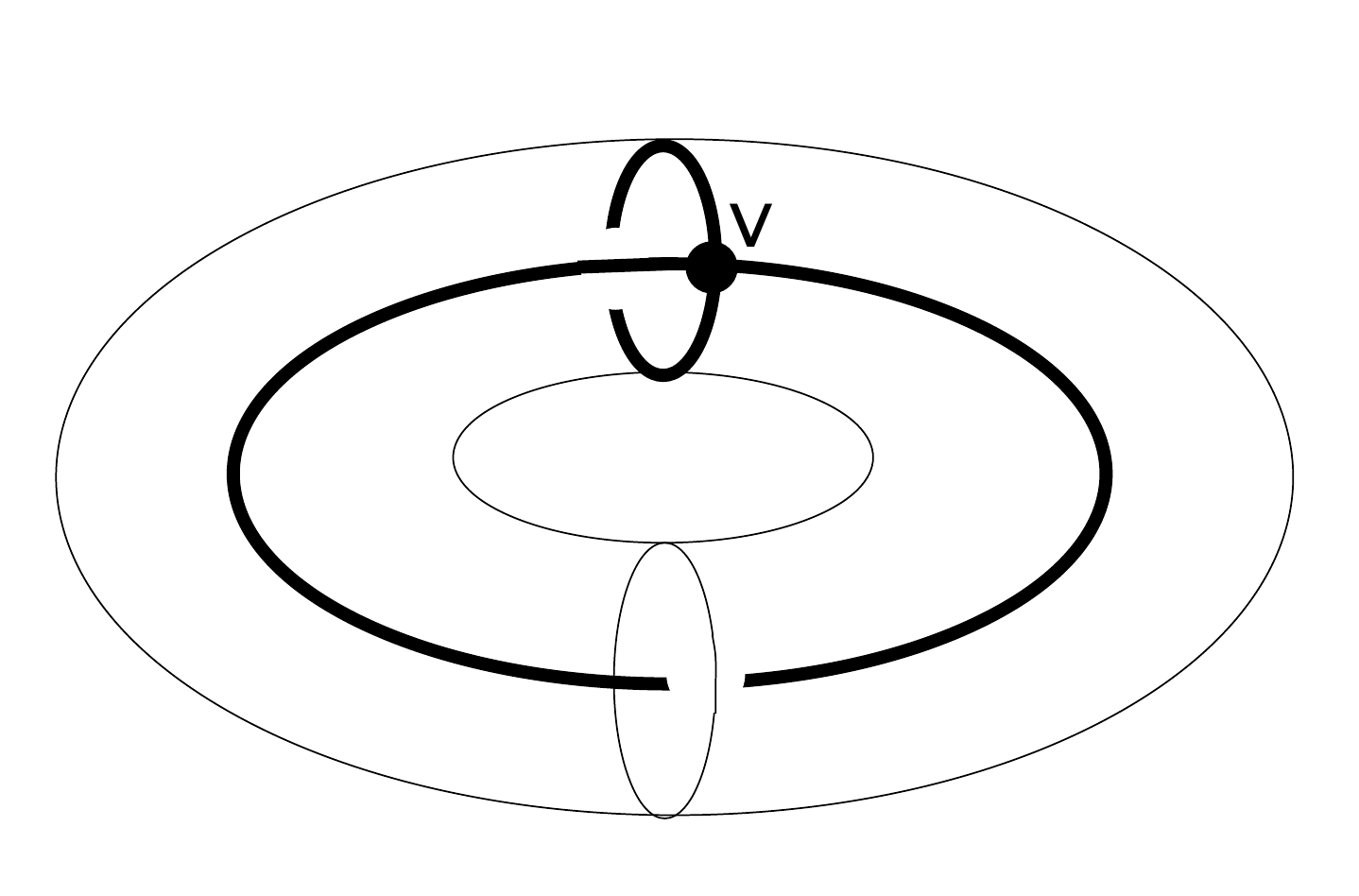}
\caption{Graph $\overline{G}$ in the Torus}
\label{gridquotient}
\end{minipage}
\hspace{0.5cm}
\begin{minipage}[b]{0.4\linewidth}
\centering
\includegraphics[width=\textwidth]{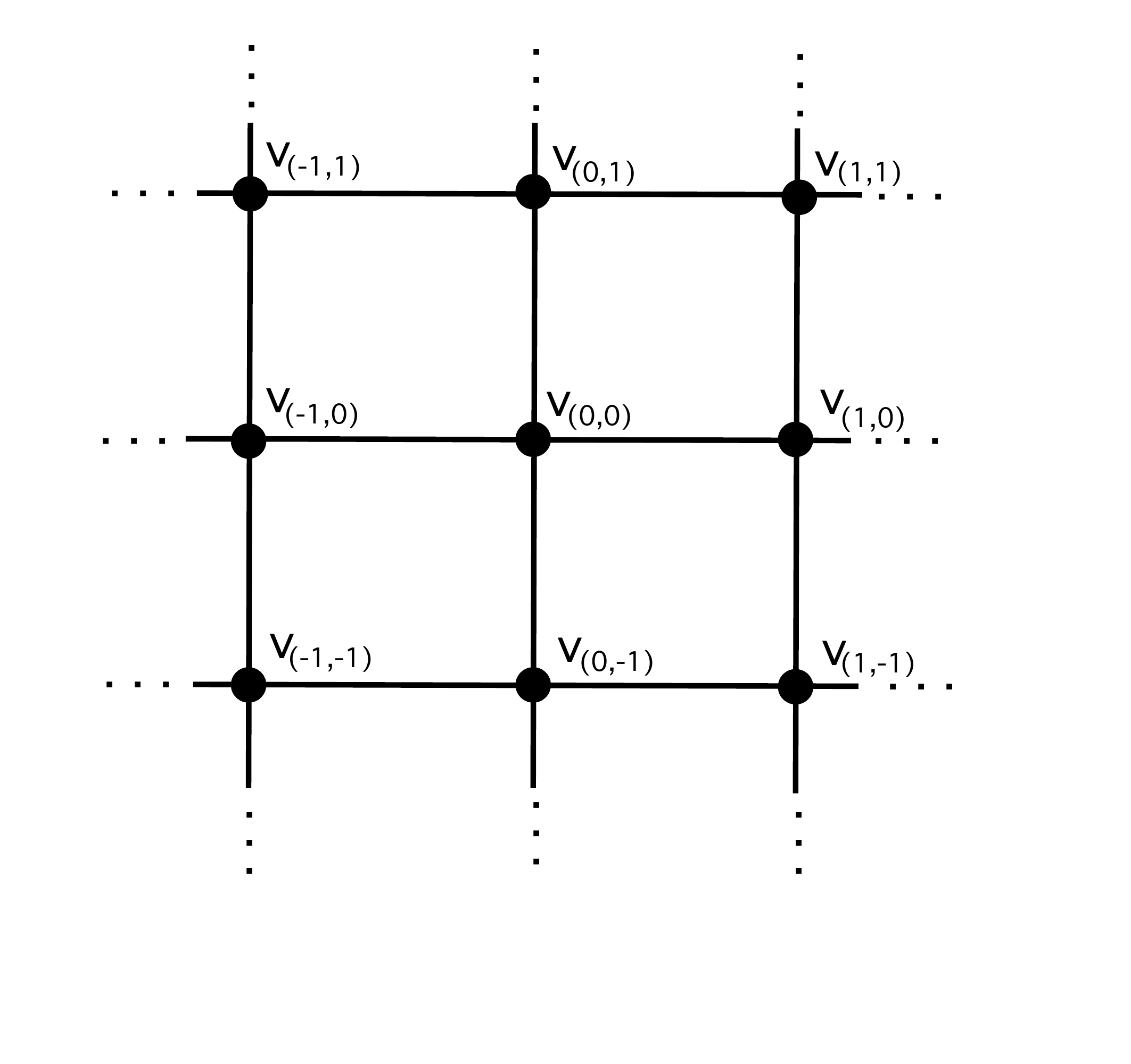}
\caption{Grid Graph $G$}
\label{grid}
\end{minipage}
\end{center}
\end{figure}\end{example}

\begin{example} \label{mitsu}  Consider the Mitsubishi (three diamond) graph $G$ with vertices labeled as in Figure \ref{mitsu}. The Laplacian matrix is 
\begin{equation*} L= \begin{pmatrix} 6 & -1-x^{-1}-y^{-1} & -1 - y^{-1}- x y^{-1}\\ -1-x-y & 3 & 0 \\ -1 - y -x^{-1} y & 0 & 3 \end{pmatrix}. \end{equation*}
The Laplacian polynomial is $\De_0(x, y) = 6(6 - x - x^{-1} - y - y^{-1} - x y^{-1} - x^{-1} y)$. The polynomial vanishes modulo $2$ since the medial graph $M(G)$ has closed components.

\begin{figure}
\begin{center}
\includegraphics[height=3.1 in]{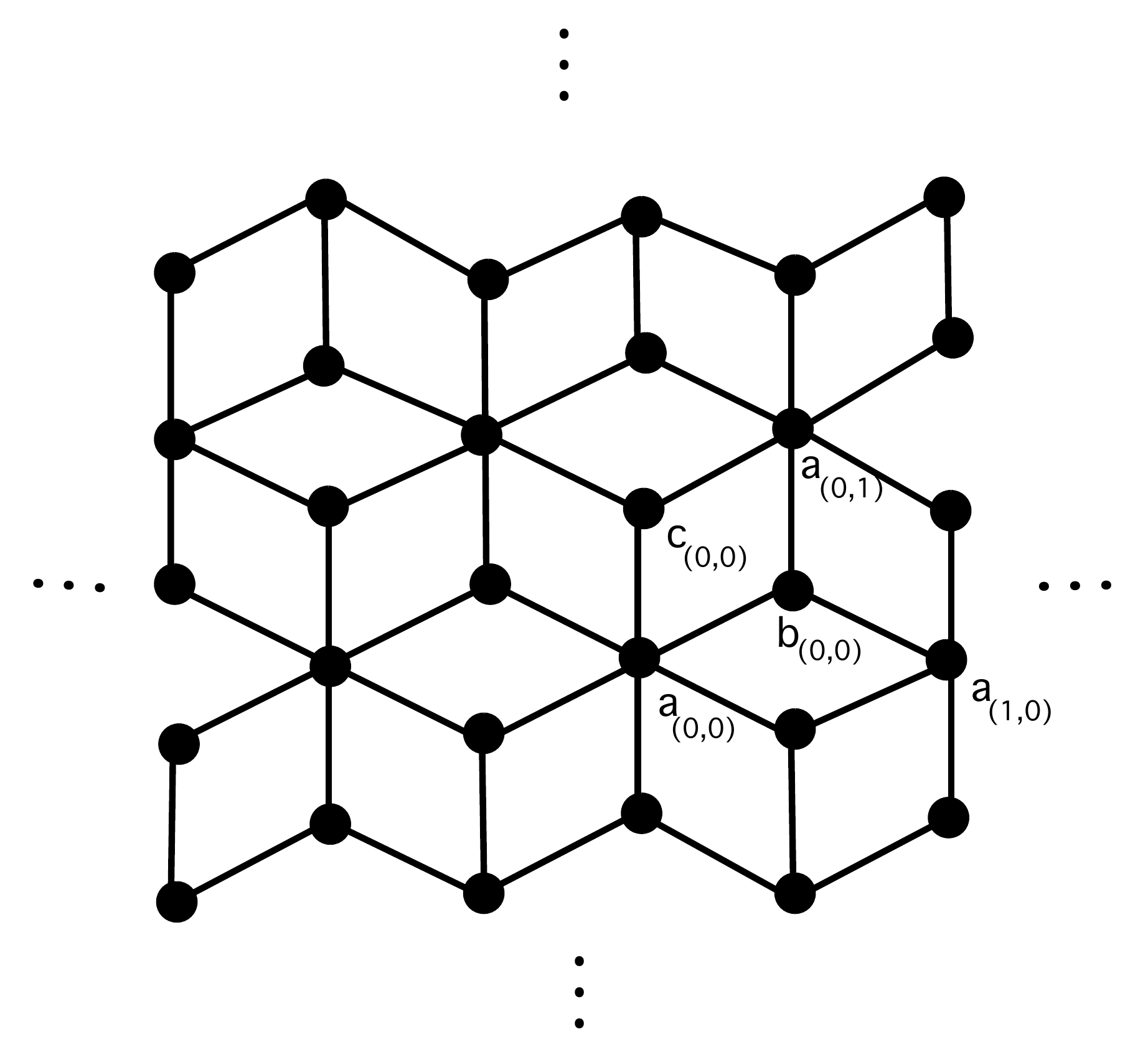}
\caption{Mitsubishi Graph $G$}
\label{mitsu}
\end{center}
\end{figure}
\end{example}

%\vfil \eject

%%%%%%%%%%%%%%%%%%%%%%%%%

\section{Mahler measure and spanning trees}  \label{mahler} We review some basic notions of algebraic dynamics as applied to locally finite graphs. A general treatment can be found in  \cite{EW99} or \cite{Sc95}.

Assume that $G$ is a graph, not necessarily planar, that admits a cofinite free $\Z^d$-symmetry. We assume that the vertices and edges of $G$ consist of the orbits of finitely many vertices $v_1, \ldots, v_n$ and edges $e_1, \ldots, e_m$, respectively. For any $x_1^{s_1}\cdots x_d^{s_d} \in \Z^d$, we denote  the vertex $(x_1^{s_1}\cdots x_d^{s_d}) v_i$ by $v_{i, \s}$, where $\s= (s_1, \ldots, s_d)$. We use similar notation for edges. 

Let $\overline{G}$ be the quotient graph of $G$. By abuse of notation, we denote its vertices by $v_1, \ldots, v_n$ and edges by $e_1, \ldots, e_m$. 
We regard $G$ as a covering graph of $\overline{G}$. If $\La \subset \Z^d$ is any subgroup, then we denote by $G_\La$  the intermediate covering graph of $\overline{G}$.  If $\La$ has finite index $r$, then $G_\La$ is a finite, $r$-sheeted covering graph. If $G$ is connected, then $G_\La$ is connected for every subgroup $\La$. 

The vector space $\C$ of conservative vertex colorings is the dual space $\text{Hom}(C, \F)$, where $C$ is the finitely generated module with presentation matrix $L$. We regarded $L$ as a matrix over $\F[x_1^{\pm 1}, \ldots, x_d^{\pm 1}]$. However, we can also regard it over the ring $\Rd=\Z[x_1^{\pm 1}, \ldots, x_d^{\pm 1}]$
of Laurent polynomials in $d$ variables with integer coefficients.
By our assumptions, $C$ is a finitely generated $\Rd$-module. 

We replace the field $\F$ with the additive circle group $\T= \R/\Z$. Then 
$\text{Hom}(C, \T)$ is the \emph{Pontryagin dual} $\hat C$ of $C$. A homomorphism is a function $\rho$ that assigns a ``color" $\a_{i, \s}\in \T$ to each vertex $v_{i, \s}$ in such a way that, when extended linearly, all $\Z^d$-multiples of row vectors of $L$ are mapped to zero. Clearly, $\hat C$ is an abelian group under coordinate-wise addition. 

We endow $C$ with the discrete topology, and the space of homomorphisms with the compact-open topology. Then $\hat C$ is a compact, 0-dimensional topological group. Moreover, the module actions of $x_1, \ldots, x_d$ determine commuting homeomorphisms $\si_1, \ldots, \si_d$ of $\hat C$. Explicitly, if $\rho$ assigns $\a_{i, \s}$ to $v_{i, \s}$, then $\si_j \rho$ assigns $\a_{i, \s'}$, where $\s'$ is obtained by adding 1 to the $j$th component of $\s$. Consequently, $\hat C$ has a $\Z^d$-action $\si: \Z^d \to \text{Aut}(\hat C)$. We denote $\si(\s)$ by $\si_\s$.

\begin{definition} Let $\La \subset \Z^d$ be a subgroup. A $\La$-\emph{periodic point} of $\hat C$ is homomorphism $\rho$ such that $\si_\s \rho = \rho$, for any $\s \in \La$.
\end{definition} 

The set of all $\La$-periodic points is a subgroup of $\hat C$, denoted here by $\text{Per}_\La(\si)$. 

\begin{definition} Let $G$ be a finite graph with $r$ connected components
$G_1, \ldots, G_r$. The \emph{complexity} $T(G)$ is the product $\t_1 \cdots \t_r$, where $\t_i$ is the number of spanning trees of $G_i$. 

\end{definition} 

\begin{remark} There are many ways to define complexity of a graph. The quantity $T(G)$ is the number of spanning forests of $G$ having minimal number of component trees. \end{remark} 

\begin{prop} \label{counting} Let $\La \subset \Z^d$ be a subgroup. Then $\text{Per}_\La(\si)$ is isomorphic to the group of conservative vertex colorings of $G_\La$. If $\La$ has finite index, then $\text{Per}_\La(\si)$ consists of $T(G_\La)$ tori, each having dimension equal to the number of connected components of the quotient graph $G_\La$.
\end{prop} 

\begin{proof} 
There is a natural isomorphism, which is also a homeomorphism, between $\text{Per}_\La(\si)$ and the group of conservative vertex colorings of $G_\La$. The latter is a subspace of $\T^{V_\La}$, where $V_\La$ is the vertex set of $G_\La$. It can be computed from the Laplacian matrix of $G_\La$, a diagonal block matrix in which each block is the Laplacian matrix of a  component of $G_\La$. Then $\text{Per}_\La(\si)$ is the Cartesian product of the corresponding spaces. It suffices to show that the space corresponding to the $ith$ component of $G_\La$ consists of $\t_i$ pairwise disjoint circles, where $\t_i$ is the number of spanning trees of the component.

Consider the $i$th component and corresponding block in the Laplacian matrix. The block has corank  1. By Kirchhoff's matrix-tree theorem (see, for example, Chapter 13 of \cite{GR01}), the absolute value of the determinant of any submatrix obtained by deleting a row and column is equal to $\t_i$. The block is a presentation matrix for a finitely generated abelian group of the form $A \oplus \Z$, where $|A| = \t_i$. The dual group is isomorphic to $A \times \T$. Topologically it consists of $\t_i$ pairwise disjoint circles.  \end{proof}

The \emph{topological entropy} $h(\si)$ of our $\Z^d$-action $\si$  is a measure of complexity. The general definition can be found in \cite{EW99} or \cite{Sc95}. By a fundamental result of D. Lind, K. Schmidt and T. Ward, \cite{LSW90} \cite{Sc95}, it is equal to the exponential growth rate of $|\text{Per}_\La(\si)|$, the  number of components of  $\text{Per}_\La(\si)$, using a suitable sequence of subgroups $\La$:

\begin{equation*}\label{growth} h(\si) =  \limsup_{\langle \La \rangle \to \infty} \frac{1}{|\Z^d/\La|}
\log |\text{Per}_\La(\si)|. \end{equation*}

Here $\langle \La \rangle$ is the minimum length of a nonzero element of $\La$. Heuristically, the condition that $\langle \La \rangle$ tends to infinity ensures that the sublattice $\La$ of $\Z^d$ grows in all directions as we take a limit. 

There is a second way to compute $h(\si)$, which uses Mahler measure.

\begin{definition} \label{mahler} The \emph{logarithmic Mahler measure} of a nonzero polynomial 
$f(x_1, \ldots, x_d) \in \Rd$ is 
\begin{equation*} m(f) = \int_0^1 \ldots \int_0^1 \log|f(e^{2\pi i \theta_1}, \ldots, e^{2\pi i \theta_d})| d\theta_1 \cdots d\theta_d. \end{equation*}

\end{definition} 

\begin{remark} (1) The integral in Defintion \ref{mahler} can be singular, but nevertheless it converges. \smallskip

(2) When $d=1$, Jensen's formula shows that $m(f)$ can be described another way. If 
$f(x) = c_s x^s+ \cdots c_1 x + c_0$, $c_0c_s \ne 0$, then
\begin{equation*} m(f) = \log|c_s| + \sum_{i=1}^s \log |\lambda_i|, \end{equation*}
where $\lambda_1, \ldots, \lambda_s$ are the roots of $f$.  \smallskip

(3) If $f, g \in \Rd$, then $m(fg) = m(f) + m(g)$. Moreover, $m(f) =0$ if and only if $f$ is the product of 1-variable cyclotomic polynomials, each evaluated at a monomial of $\Rd$
(see \cite{Sc95}). 

\end{remark}

By \cite{Sc95} (see Example 18.7), $h(\si)$ is equal to the logarithmic Mahler measure $m(\De_0)$ of the $0$th Laplacian polynomial of $G$.  By Proposition \ref{counting}, we have: 

\begin{theorem} \label{trees} Let $G$ be  graph with cofinite free $\Z^d$-symmetry. Then 
$$m(\De_0)= \limsup_{\langle \La \rangle \to \infty} \frac{1}{|\Z^d/\La|}
\log T(G_\La), $$
where $T(G_\La)$ is the
complexity of the covering graph $G_\La$. When $d=1$, the limit superior can be replaced by an ordinary limit. 

\end{theorem}

\begin{remark}  The statement about $d=1$ is shown in \cite{LSW90}. 
When $d>1$ and $G$ is connected, the limit superior can again be replaced with an ordinary limit.  This was proved by Lyons in \cite{Ly05} using measure theoretic techniques. The stronger result  follows also from the algebraic dynamical approach here, as we show in \cite{SW15}. \end{remark} \bigskip

{\sl For the remainder of the section we assume that $G$ is connected.}\\

Let $\La$ be a finite-index subgroup of $\Z^d$ and let $R=R(\La)$ a fundamental domain.  Let $G|_R$ denote the full subgraph on the vertices with indices in $R$. The number of such vertices is $s = k |\Z^d/\La \Z^d|$, where $k$ is the number of vertex orbits of $G$. Taking a limit over increasingly large domains $R$ such that $G|_R$ is connected, one defines the \emph{thermodynamic limit} (also called the \emph{bulk limit} or \emph{spanning tree constant})
\begin{equation*} \lambda_G = \lim_{s \to \infty} \frac{1}{s}\log \t(G|_R), \end{equation*}
where $\t$ is the number of spanning trees. In the literature, $R$ is usually chosen to be a $d$-dimensional cube.

\begin{remark} In the examples of lattices most often considered, $G|_R$ is connected for every rectangular fundamental domain $R$.  However, in general $G|_R$ need not be connected. Examples are easy to construct. \end{remark}

Partition functions and other analytic tools have been used to compute growth rates of the number of spanning trees; see, for example,  \cite{Wu77}, \cite{SW00}, \cite{CS06} and \cite{TW10}. In \cite{BP93}, Burton and Pemantle  considered an {\sl essential spanning forest process} for locally finite graphs with free $\Z^d$-symmetry.  This process is a weak limit of uniform measure on the set of spanning trees of $G|_R$ as $\<\La\>\to \infty$, and its (measure-theoretic) entropy is seen to be $k\lambda_G$.  For the case $k=1$,  R. Solomyak \cite{So98} proved by analytic methods that this entropy is equal to the Mahler measure of the polynomial we have called $\De_0$.

\begin{theorem}\label{thermo} The sequences $\{\t(G_\La)\}$ and $\{\t(G|_R)\}$ have the same exponential growth rate as $\<\La\>\to\infty$, provided each $G|_R$ is connected.  Thus 
$\lambda_G = \frac{1}{k}m(\De_0)$.
\end{theorem}

\begin{remark} The recognition that asymptotic complexity can be measured by considering either quotients or subgraphs is not new. A very general result is appears in \cite{Ly05} (see Theorem 3.8). The proof below for the graphs that we consider is relatively elementary. \end{remark} 

\begin{proof}
Since every spanning tree of $G|_R$ can be viewed as a spanning tree for $G_\La$, we see immediately that $\lambda_G \le\frac{1}{k}m(\De_0)$.

Let $T$ be a spanning tree for $G_\La$. We may also regard $T$ as a periodic spanning tree for $G$.  Its restriction $T_R$ to $G|_R$ might not be connected. However, when $d>1$, within a slightly larger domain $R'$ containing $R$ we can extend $T_R$ to a spanning tree $T'$.  We take $R'$ to consist of elements of $\Z^d$ that are some bounded distance from $R$, so that every edge of $G$ with a vertex in $R$ has its other vertex in $R'$, and the graph $G|_{R'\setminus R}$ is connected.   (This is where $d>1$ is needed.)  Then the number $s'$ of vertices of $R'$ satisfies $s'/s \to 1$ as $s \to \infty$. We can choose $T'$ to contain all the edges of $T$ with at least one vertex in $G|_R$, thereby ensuring that $T \mapsto T'$ is an injection. This gives the reverse inequality
$\lambda_G \ge \frac{1}{k}m(\De_0)$.

In the case $d=1$, $G|_{R'\setminus R}$ consists of two connected components.  Applying the above construction, we can obtain a graph $T'$ that is either a spanning tree or a two-component spanning forest for $G|{R'}$.  Any such forest can be obtained from a spanning tree by deleting one of the $s'-1$ edges.  Hence $\tau(G_\La)$ is no more than $s'$ times the number of spanning trees of $G|_R'$.  This rough upper bound suffices to give the desired growth rate.
\end{proof}

\begin{prop} \label{estimate1} (Cf. \cite{CS04}) Let $G$ be a locally finite connected graph with cofinite free $\Z^d$-symmetry.  Then 
\begin{equation*} m(\De_0) \le |\overline{V}| \log \frac{2 |\overline{E}|}{|\overline{V}|},\end{equation*}
where $\overline{V}, \overline{E}$ are the vertex and edge sets, respectively, of $\overline{G}$. \end{prop}

\begin{proof} By a result of G.R. Grimmett \cite{Gr76}, for every finite-index sublattice $\La$ of $\Z^d$
\begin{equation*} \t(G_\La) \le \frac{1}{|V_\La|} \bigg( \frac{2 |E_\La|}{|V_\La| -1}\bigg)^{|V_\La| -1}, \end{equation*}
where $G_\La=(V_\La,E_\La)$. Letting $r = |\Z^d/\La|$, we have
$$m(\De_0) \le  \limsup_{r \to \infty} \frac{1}{r} \log
\frac{1}{r|\overline{V}|} \bigg( \frac{2r |\overline{E}|}{r|\overline{V}| -1}\bigg)^{r|\overline{V}| -1}.$$
The result follows by elementary analysis. 
\end{proof}

\begin{example} For the ladder graph in Example \ref{ladder}, 
$$m(\De_0(x)) = m(x^2-4x +1) = \log(2 + \sqrt{3}) \approx 1.317.$$
The thermodynamic limit $\lambda_G\approx 0.658$ was computed in \cite{SW00}.
The upper bound of Proposition \ref{estimate1} (with $|\overline{V}|= 2, |\overline{E}|=3$) is $2.198$.  \end{example}

\begin{example} For the grid in Example \ref{2grid}, $\lambda_G=m(4-x-x^{-1}-y-y^{-1})\approx 1.166$. The upper bound of Proposition \ref{estimate1} (with $|\overline{V}|=1, |\overline{E}|=2$) is $1.386$. 
\end{example}

\begin{figure}
\begin{center}
\includegraphics[height=2 in]{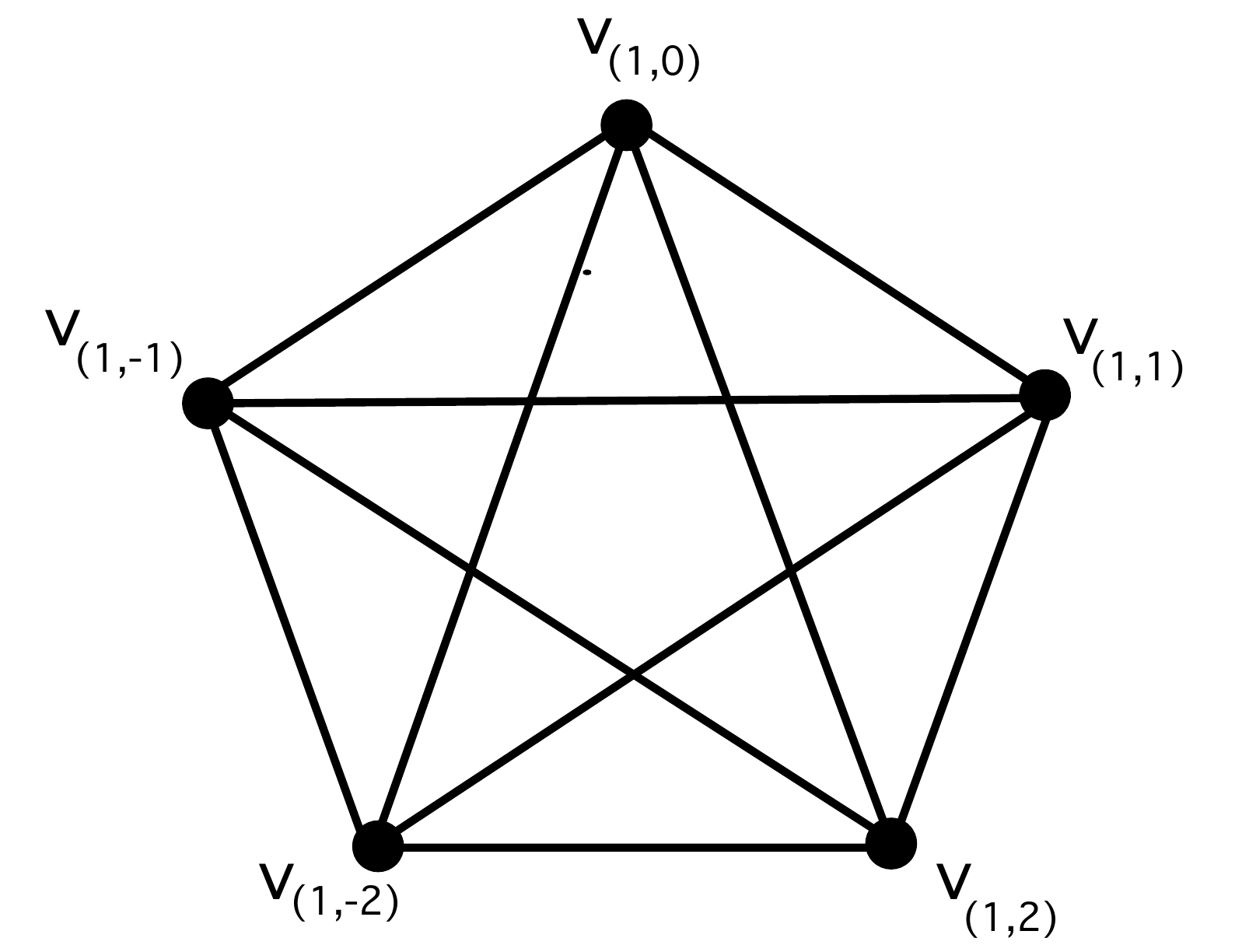}
\caption{Circulant Graph $C_5^{1,2}$}
\label{circulant}
\end{center}
\end{figure}

\begin{figure}
\begin{center}
\includegraphics[height=3 in]{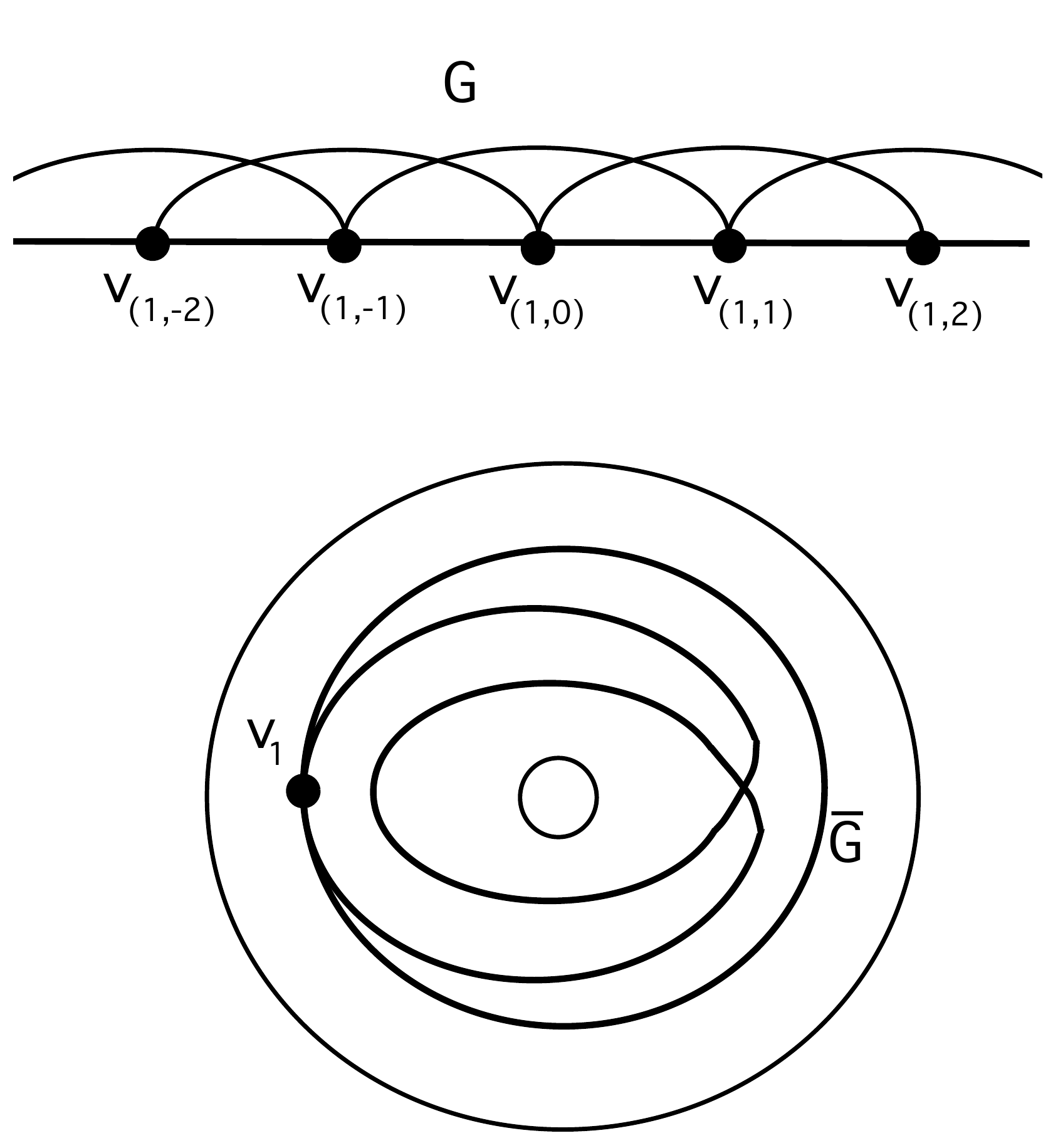}
\caption{Circulant Cover $G$ and Quotient $\overline{G}$}
\label{circcov}
\end{center}
\end{figure}

\begin{example}  \label{circulant} A \emph{circulant graph} $C_n^{s_1, \ldots, s_k}$ is a $2k$-regular graph with $n$ vertices $v_1, \ldots, v_n$ such that $v_i$ is adjacent to $2k$ vertices
$v_{i \pm s_1}, \ldots, v_{i \pm s_k}$, where indices are taken modulo $n$.  Several authors \cite{LPW97}, \cite{GYZ10} have investigated the growth rate of $\t(C_n^{s_1, \ldots, s_k})$ using a blend of combinatorics and analysis. We recover the growth rates
very quickly with algebraic methods. 

Let $\La = (n) \subset \Z$. The graph $C_n^{s_1, \ldots, s_k}$ can be regarded as an $n$-sheeted cover $G_\La$ of a graph $\overline{G}$ with a single vertex $v$ and $2k$ edges immersed in the annulus. The $j$th edge winds $s_j$ times. A simple example appears in Figures \ref{circulant} and \ref{circcov} below. 

It is immediate that the Laplacian polynomial is $$\De_0(x) = 2k -  \sum_{j=1}^k (x^{s_j}+x^{-s_j}).$$
The  growth rate $\lim_{n\to \infty} \frac{1}{n} \log \t(C_n^{s_1, \ldots, s_k})$ is equal to the logarithmic
Mahler measure $m(\De_0)$. 

When $s_1=1, s_2=2$ and $s_j =0$ for $j>2$, as in Figure \ref{circulant}, $\De_0(x) = 4 - x - x^{-1} - x^2-x^{-2} = (x-1)^2(x^2+3x+1)$ and 
$\lim_{n\to \infty} \frac{1}{2}\t(C_n) = (3+\sqrt 5)/2$, as obtained in \cite{GYZ10} (see p. 795).

In general, 
$$\lim_{n\to \infty}  \frac{1}{n} \log \t(C_n^{s_1, \ldots, s_k})= m(\De_0)$$
$$ = \int_0^1 \log \Big| 2 k - \sum_{j=1}^k (e^{2 \pi i s_j\theta} + e^{-2 \pi i s_j\theta})\Big| d\theta$$
$$ =  \int_0^1  \log  \Big(\sum_{j=1}^k (2- 2 \cos 2 \pi s_j \theta) \Big) d\theta$$
$$= \int_0^1  \log  \Big(4  \sum_{j=1}^k  \sin^2  \pi s_j \theta  \Big) d\theta$$
$$=\log 4 +   \int_0^1 \log  \Big( \sum_{j=1}^k \sin^2 \pi s_j \theta  \Big) d\theta,$$
which agrees with Lemma 2 of \cite{GYZ10}. 

We conclude with a comment about Theorem 6 of \cite{GYZ10}, which states: 
$$\lim_{s_k \to \infty} \ldots \lim_{s_1 \to \infty} \lim_{n \to \infty} \frac{1}{n} \log \t(C_n^{s_1, \ldots, s_k}) = \log 4+ \int_0^1 \cdots \int_0^1 \log\Big( \sum_{j=1}^k \sin^2 \pi \theta_j \Big) d\theta_1 \cdots d\theta_k.$$ In view of Definition \ref{mahler}, this integral is simply the Mahler measure of the $k$-variable polynomial that is obtained from $\De_0(x)$ by replacing each $x^{s_j}+x^{-s_j}$ with $x_j+ x_j^{-1}$.  Thus Theorem 6 is an instance of the general limit formula
$$\lim_{s_k \to \infty} \ldots \lim_{s_1 \to \infty}m(f(x, x^{s_1},\ldots, x^{s_k}))=m(f(x_0,x_1, \ldots, x_k)),$$
which is found in Appendix 4 of \cite{Bo81}.

\end{example}

\bigskip

\ni Department of Mathematics\\
\ni University of Southern Mississippi\\Hattiesburg, MS 39406 USA\\
\ni Email: Kalyn.Lamey@usm.edu\\ 

\ni Department of Mathematics and Statistics,\\
\ni University of South Alabama\\ Mobile, AL 36688 USA\\
\ni Email: silver@southalabama.edu, swilliam@southalabama.edu

\begin{thebibliography}{1}

%\bibitem{Bi93} N.L. Biggs, Algebraic Graph Theory, 2nd ed., Cambridge Univ. Press, 1993.

\bibitem{BR03} C.P. Bonnington and R.B. Richter, Graphs embedded in the plane with a bounded number of accumulation points, J. Graph Theory 442 (2003), 132--147.

\bibitem{Bo81} D. Boyd, Speculations concerning the range of Mahler's measure, Canad. Math. Bull. 24 (1981), 453--469.

\bibitem{BKR07} M. Braverman, R. Kulkarni and S. Roy, Parity problems in planar graphs, Electronic Colloquium on Computational Complexity, Report No. 35 (2007), 1--26. 


\bibitem{BKW09} H. Bruhn, S. Kosuch and M.W. Myint, Bicycles and left-right tours in locally finite graphs, Europ. J. Comb. 30 (2009), 356--371. 

\bibitem{BP93} R. Burton and R. Pemantle, Local characteristics, entropy and limit theorems for spanning trees and domino tilings via transfer-impedances, Ann. Prob. 21 (1993), 1329--1371.

\bibitem{CSW14} J.S. Carter, D.S. Silver and S.G. Williams, Three dimensions of knot coloring, American Math. Monthly 121 (2014), 506--514. 

\bibitem{CS04} S.-C. Chang and R. Shrock, Tutte polynomials and related asymptotic limiting functions for recursive families of graphs, Advances in Appl. Math. 32 (2004), 44--87.

\bibitem{CS06} S.-C. Chang and R. Shrock, Some exact results for spanning trees on lattices. J. Phys. A, 39(20):5653--5658, 2006.

\bibitem{CF63} R.H. Crowell and R.H. Fox,  Introduction to Knot Theory, Ginn and Company, Boston 1963.

\bibitem{EW99} G. Everest and T. Ward, Heights of Polynomials and Entropy in Algebraic Dynamics, Springer-Verlag, London 1999. 


\bibitem{Fo93} R. Forman, Determinants of Laplacians on graphs, Topology, 32 (1993), 35--46. 

\bibitem{GNT03} A. Garcia, M. Noy, J. Tejel, The asymptotic number of spanning trees in d-dimensional square lattices, J. Combin. Math. Combin. Comput. 44 (2003),.
109--113.


\bibitem{GR01} C. Godsil, G. Royle, Algebraic Graph Theory, Springer Verlag, 2001.

\bibitem{Je89} F. Jaeger, Tutte polynomials and bicycle dimension of ternary matroids, Proc. Amer. Math. Soc. 107 (1989), 17--25.

\bibitem{GYZ10} M.J. Golin, X. Yong and Y. Zhang, The asymptotic number of spanning trees in circulant graphs, Discrete Mathematics 310 (2010), 792--803.
 
\bibitem{Gr76} G.R. Grimmett, An upper bound for the number of spanning trees of a graph, Discrete Math. 16 (1976), 323--324.

\bibitem{GR12} A.J. Guttmann and M. Rogers, Spanning tree generating functions and Mahler measures, J. Phys. A: 45 (2012), n 49, 494001, 24 pp. 

\bibitem{JDT09} X. Jin, F Dong and E.G. Tay, On graphs determining links with maximal number of components via medial construction, Discrete Appl. Math. 157 (2009), 3099--3110.

\bibitem{Ka83} L.H. Kauffman, Formal Knot Theory, Princeton University Press, 1983. 
\bibitem{Ke11} R. Kenyon, Spanning forests and the vector bundle Laplacian, Annals of Probability 39 (2011), 1983--2017.
\bibitem{Ke12} R. Kenyon, The Laplacian on planar graphs and graphs on surfaces, in
Current Developments in Mathematics (2011),  1 -- 68. 

\bibitem{La83} W.M. Lawton, A problem of Boyd concerning geometric means of polynomials, J. Number Theory 16 (1983), 356--362. 

\bibitem{LSW90} D.A. Lind, K. Schmidt and T. Ward, Mahler measure and entropy for commuting automorphisms of compact groups, Inventiones Math. 101 (1990), 593--629.

\bibitem{LPW97}   Z. Lonc, K. Parol, J.M. Wojciechowski, On the asymptotic behavior of the maximum number of spanning trees in circulant graphs, Networks 30 (1)
(1997), 47--56.

\bibitem{Ly05} R. Lyons, Asymptotic enumeration of spanning trees, Combinatorics, Probability and Computing 14 (2005), 491--522. 

\bibitem{My09} M.W. Myint, Bicycles and left-right tours in locally finite graphs, doctoral dissertation submitted to University of Hamburg, 2009.

\bibitem{Mp02} E.G. Mphako, The component number of links from graphs, Proc. Edinburgh Math. Soc. (2002) 45, 723 -- 730.  

%\bibitem{Pe91} R. Pemantle, Choosing a spanning tree of the integer lattice uniformly, Ann. Prob. 19 (1991), 1559--1574.

\bibitem{RV08} R.B. Richter and A. Vella, Cycle spaces in topological spaces, J. of Graph
Theory 59 (2008), 115--144.  

\bibitem{Sc95} K. Schmidt, Dynamical Systems of Algebraic Origin, Birkh\"auser, Basel, 1995. 


\bibitem{Sh75} H. Shank, The theory of left-right paths, in Combinatorial Mathematics III, Lect. Notes Math. 452, Springer-Verlag, 1975, 42--54.

\bibitem{SW00} R. Shrock and F.Y. Wu, Spanning trees on graphs and lattices in $d$ dimensions, J. Phys. A: Math. Gen. 33 (2000), 3881--3902. 

\bibitem{SW02} D.S. Silver and S.G. Williams, Mahler measure, links and homology growth, Topology 41 (2002), 979--991. 

\bibitem{SW14} D.S. Silver and S.G. Williams, On the component number of links from plane graphs, Journal of Knot Theory and its Ramification  24 (2015), 1520002, 5 pp.

\bibitem{SW15} D.S. Silver and S.G. Williams, Spanning trees and Mahler measure, preprint, 2015. 

\bibitem{So98} R. Solomyak, On coincidence of entropies for two classes of dynamical systems, Ergod. Th. \& Dynam. Sys. 18 (1998), 731--738.

\bibitem{TW10} E. Teufl and S. Wagner, On the number of spanning trees on various lattices, Journal of Physics A: Mathematical and Theoretical 43, 415001, 2010.

\bibitem{Tr85} L. Traldi, On the Goeritz matrix of a link, Math. Z. 188 (1985), 203--213.

\bibitem{Wu77} F.Y. Wu, Number of spanning trees on a lattice. J. Phys. A, 10(6)  L113, 1977.

\end{thebibliography}
\end{document}